\documentclass{amsart}

\usepackage{amssymb, amsmath, latexsym, amsfonts, amsthm, mathrsfs, mathtools} 
\usepackage{verbatim} 

\usepackage{geometry}
\usepackage{hyperref} 
\usepackage{color}
\usepackage{bbm} 
\usepackage{appendix}

\usepackage[pagewise]{lineno} 

\newtheorem{proposition}{Proposition}[section]
\newtheorem{lemma}[proposition]{Lemma}
\newtheorem{corollary}[proposition]{Corollary}
\newtheorem{theorem}[proposition]{Theorem}
\newtheorem{example}[proposition]{Example}

\theoremstyle{definition}
\newtheorem{definition}[proposition]{Definition}
\newtheorem{remark}[proposition]{Remark}

\newtheorem{assumption}[proposition]{Assumption}

\newcommand{\Ito}{It\^o }

\newcommand{\N}{\mathbb{N}}
\newcommand{\R}{\mathbb{R}}
\newcommand{\PP}{\mathbb{P}}
\renewcommand{\P}{\PP}
\newcommand{\F}{\mathcal{F}}

\newcommand{\E}{\mathbb{E}}

\newcommand{\e}[1]{\mathrm{e}^{#1}}

\newcommand{\loc}{\mathrm{loc}}

\newcommand{\dd}{\mathrm{d}}
\renewcommand{\d}{\dd}

\DeclareMathOperator{\law}{Law}

\DeclareMathOperator*{\esssup}{ess\,sup}

\DeclarePairedDelimiter{\norm}{\|}{\|}

\newcommand{\1}{\mathbbm{1}}

\newcommand{\vertiii}[1]{{\vert\kern-0.25ex\vert\kern-0.25ex\vert #1
		\vert\kern-0.25ex\vert\kern-0.25ex\vert}}

\title{Comparison Principles for Stochastic Volterra equations}
\author{Ole Ca\~nadas}
\address[Ole Ca\~nadas]{School of Mathematical Sciences, Dublin City University, Ireland}
\curraddr{}
\email{ole.canadas2@mail.dcu.ie}

\author{Martin Friesen}
\address[Martin Friesen]{School of Mathematical Sciences, Dublin City University, Ireland}
\curraddr{}
\email{martin.friesen@dcu.ie}

\date{\today}

\thanks{O.C. is funded by the Irish Research Council grant GOIPG/2023/3129 and was previously supported by the ERASMUS Mobility project during which part of this research was carried out.}

\begin{document}

\begin{abstract}
    In this work, we establish a comparison principle for stochastic Volterra equations with respect to the initial condition and the drift $b$ applicable to a wide class of Volterra kernels and input curves $g$. Such input curves are allowed to be singular in zero, and appear, e.g., in Markovian lifts for Volterra equations. For completely monotone kernels, our result holds without any further restrictions, while for regular kernels we give a characterisation of the comparison principle. Finally, we show that for not completely monotone kernels such a principle fails unless the drift is monotone. As a side-product of our results, we also complement the literature on the weak existence of continuous nonnegative solutions, which covers the rough Cox-Ingersoll-Ross process with singular initial conditions.
\end{abstract}

\maketitle

\noindent
\textbf{Keywords:} stochastic Volterra process; comparison principle; rough Cox-Ingersoll-Ross process; splitting method\\



\section{Introduction}

Let $B$ be a standard one-dimensional Brownian motion, $0\neq K\in L^2_\loc(\R_+)$ a Volterra convolution kernel, $b,\sigma\colon\R_+\times\R\longrightarrow \R$ continuous with linear growth in the spatial variable, and $g \in L_{\loc}^2(\R_+)$. In this work, we study stochastic Volterra equations (SVEs) of the form
\begin{equation}\label{eq:sve}
     X_t= g(t)+\int_0^t K(t-s)b(s,X_s)\,\d s + \int_0^t K(t-s)\sigma(s,X_s)\,\d B_s.
\end{equation}
In many applications, one assumes that the Volterra kernel is given by the fractional Riemann-Liouville kernel $K(t) = \frac{t^{\alpha-1}}{\Gamma(\alpha)}$ with $\alpha \in (1/2,2)$. Such equations have gained increased attention in mathematical Finance for the modelling of rough volatility, see \cite{pricing_under_rough_volatility,affine_fractional_stochastic_volatility_models,el_euch1,el_euch2, short_time_fukasawa}.

The weak existence of a solution was established in \cite{affine_volterra} whenever $b,\sigma$ are continuous with linear growth, $g$ is constant and $K$ satisfies an additional regularity condition given below. An extension to processes with jumps and more general $g$ is discussed in \cite{weak_solution}. When $b$ and $\sigma$ are Lipschitz continuous, the existence of a unique strong solution was shown in \cite[Theorem A.1]{markovian_structure}, see also \cite{berger_mizel, protter_volterra}. For non-Lipschitz coefficients, the pathwise uniqueness of solutions is more subtle and yet not fully understood. For the fractional Riemann-Liouville kernel $K(t) = \frac{t^{\alpha-1}}{\Gamma(\alpha)}$ with $\alpha \in (1/2,1)$, the method proposed in \cite{mytnik} gives the existence and uniqueness of strong solutions whenever $\sigma$ is $\eta$-H\"older continuous and $\alpha \eta > \frac{1}{2}$. An extension of this method towards general drifts and non-convolution equations was recently discussed in \cite{proemel_dec22}. Most recently, in \cite{hamaguchi2023weak}, the author proves the weak existence and uniqueness in law for \eqref{eq:sve} provided that $b,\sigma$ are uniformly continuous, $\sigma$ is non-degenerate, and some technical condition on the modulus of continuity holds. In contrast to the existing literature, see also \cite{proemel_scheffels_weak, proemel_dec22}, in this work we focus on continuous solutions where $g$ is not regular but belongs to $g \in L_{\loc}^{p}(\R_+) \cap C((0,\infty))$ with an additional growth condition at the origin. Such functions $g$ that are singular in $t = 0$ appear, e.g., from Markovian lifts where the initial condition does not belong to the domain of the corresponding projection operator, see e.g. \cite[Theorem 3.3]{hamaguchi2023weak}. Thus, our Theorem \ref{weakexistence_cm_lg} naturally complements the results obtained therein.

In the main part of this work, we study \emph{comparison principles} for solutions of \eqref{eq:sve}. More precisely, given initial data $(g_i,K,b_i,\sigma)$ with $i=1,2$, we provide a characterisation and sufficient conditions when the corresponding solutions $X^1,X^2$ satisfy
\[
\P[X_t^1 \leq X_t^2, \ t \geq 0] = 1.
\]
While such comparison principles are well-known for classical ordinary differential equations (see \cite{ode_walter}) and for Markovian stochastic equations (see \cite{ikeda_watanabe} for the continuous case and \cite{MR4195178} for the case of continuous-state branching processes), much less is known for their Volterra counterparts. For instance, comparison principles for deterministic Volterra equations have been studied in \cite[Chapter 12.1]{grippenberg}, while the first results on stochastic Volterra equations were obtained in \cite{sundar_comparison, tudor_comparison} for general drifts and restrictive assumptions on the diffusion which exclude our convolution setting \eqref{eq:sve}.

One key obstacle when dealing with stochastic Volterra equations stems from the absence of the semimartingale property. Indeed, while the classical methods described in \cite[Chapter VI]{ikeda_watanabe} are deeply based on the \Ito formula, in our setting an analogous approach may only be applied when $K$ is sufficiently regular in the sense that $K \in W^{1,2}_{\loc}(\R_+)$. Based on the splitting method, Alfonsi studied in \cite{alfonsi2023} a numerical approximation of stochastic Volterra equations applicable to $g(t) = x$ constant, globally Lipschitz continuous coefficients $b, \sigma$, and $K$ being either \textit{non-negativity preserving} with $K(0) > 0$ or completely monotone. The latter work includes, to our knowledge, the first general comparison principle for equations of the form \eqref{eq:sve}. This splitting method has also been recently applied in \cite{AS24} to prove the existence and uniqueness of nonnegative solutions of stochastic Volterra equations with jumps for non-Lipschitz continuous coefficients and regular Volterra kernels $K \in C^2(\R_+)$. The reader interested in other numerical schemes for stochastic Volterra processes may consult \cite{AbiJaber_ElEuch, alfonsi2022} and the references therein.

In this work, we first extend Alfonsi's numerical approximation \cite{alfonsi2023} to non-constant $g$ with a particular focus on possibly singular functions, and time-dependent coefficients. Based on this approximation, we obtain a comparison principle that covers kernels $K$ that are either \textit{non-negativity preserving} with $K(0) > 0$ or completely monotone. By approximation, we also prove the existence of a \textit{monotone coupling}. Going beyond the framework discussed in \cite{alfonsi2023}, we also study kernels $K$ that are not completely monotone, but sufficently regular. Firstly, beyond complete monotonicity, we show that a comparison principle may fails since solutions of \eqref{eq:sve} are typically not Markov processes and hence past trajectory entering the equation via $b,\sigma$ may violate the ordering of solutions. Using a support theorem for stochastic Volterra equations \cite{K21}, we provide a characterisation of the comparison principle, and discuss some necessary conditions. For the case of additive noise, we show that monotonicity of the drift is sufficient for the comparison property.

This work is organised as follows. In Section \ref{section: Non-negativity preserving case} we treat the case of completely monotone kernels, where we first prove a comparison principle for Lipschitz continuous $b,\sigma$. Then we prove the existence of continuous weak solutions for continuous $b,\sigma$ with at most linear growth, and finally derive the corresponding comparison principle in this case. As an application, we conclude Section \ref{section: Non-negativity preserving case} with a result on the weak existence of continuous nonnegative solutions of \eqref{eq:sve}. Afterwards, in section \ref{section: Regular Case}, we study similar results for the case of not necessarily completely monotone kernels that are sufficiently regular. Finally, some technical results on the existence, uniqueness, and sample path regularity of solutions to stochastic Volterra equations are collected in Appendix \ref{appendix: Existence and Uniqueness}, while the convergence of the splitting method is studied in Appendix \ref{appendix_splitting}.

\section{Non-negativity preserving case}\label{section: Non-negativity preserving case}

\subsection{Lipschitz coefficients}

In this section, we study the case where the Volterra kernel preserves nonnegativity. As a first step, let us introduce the minimal conditions on the input data $(g, K, b, \sigma)$ that will be used for the construction of continuous weak solutions.
\begin{assumption}\label{assumption}
		There exist $C>0$, $\xi \in [0,1]$ such that
        \[
            |b(t,x)|+|\sigma(t,x)|\leq C(1+|x|)^\xi
        \]
        for a.a. $t\geq0$ and $x\in\R$. Moreover, one of the following two cases holds:
        \begin{enumerate}
            \item $g\in L^\infty_\loc(\R_+)$ or $\xi = 0$ with $g \in L_{\loc}^{2}(\R_+)$. Moreover, $K\in L^2_\loc(\R_+)$ and there exist $\gamma \in (0,1/2]$ and for each $T>0$ a constant $c(T)>0$ such that
            \[
            \norm{K}_{L^2([0,h])} + \norm{K(\cdot+h)-K(\cdot)}_{L^2([0,T])} \leq c(T)h^\gamma,\quad h\in[0,1].
            \]
            \item $g\in L^{q}_\loc(\R_+)$ for some $q\in(2,\infty)$, $\xi \neq0$, there exists $\eta>0$ such that $K\in L^{2+\eta}_\loc(\R_+)$ and there exist $\gamma \in (0,1/2]$ and for each $T>0$ a constant $c(T,\eta)$ such that for $h \in [0,1]$
            \begin{equation*}
                \norm{K}_{L^{2+\eta}([0,h])} + \norm{K(\cdot+h)-K(\cdot)}_{L^{2+\eta}([0,T])} \leq c(T,\eta)h^\gamma.
            \end{equation*}
            Moreover, the value $q$ satisfies
            \begin{align}\label{eq: prime condition}
                q> 2\xi \frac{1 + \frac{1}{\eta}}{\gamma + \frac{1}{2}\frac{\eta}{2+\eta}}.
            \end{align}
        \end{enumerate}
\end{assumption}
The above assumption allows us to assume without loss of generality that for any solution $X$ of \eqref{eq:sve}, $X-g$ is continuous, c.f. Proposition \ref{prop_loc_hoelder}. Note that such a condition is satisfied by the fractional Riemann-Liouville kernel as demonstrated in the next example.

\begin{example}
    Let $K(t) = \frac{t^{\alpha-1}}{\Gamma(\alpha)}$ with $\alpha \in (1/2,1)$. Then $2+\eta < \frac{1}{1-\alpha}$, $\gamma = \alpha - 1 + \frac{1}{2+\eta}$. Hence \eqref{eq: prime condition} takes the particular form $q > 2\xi\frac{1+1/\eta}{\alpha - 1/2}$. By letting $\eta \nearrow \frac{2\alpha-1}{1-\alpha}$, we see that this condition is satisfied whenever
    \[
        q > \frac{\xi \alpha}{\left(\alpha - \frac{1}{2}\right)^2}.
    \]
\end{example}

In the same spirit, one can verify that all kernels covered by \cite[Example 2.3]{affine_volterra} satisfy the above condition. Next, to prove the desired comparison principles, let us introduce some concepts that allow us to compare the solutions. In this section, we focus on Volterra kernels that preserve nonnegativity in the sense of the following definition.

\begin{definition}\label{def_pos_preserving}
    A Volterra kernel $K\colon(0,\infty) \longrightarrow  \R_+$ preserves non-negativity if it satisfies $K(0_+)\coloneqq\lim_{t\downarrow0}K(t) \in(0,\infty)$ and for each $N \in \N$, and all $x_1,\dots,x_N\in\R$ and $0\leq t_1<\dots<t_N$ it holds that
	     \begin{equation}\label{eq:pos_preserving}
	           \sum_{\ell=1}^k x_{\ell}K(t_k-t_\ell)\geq0, \ \  \forall k\in\{1,\dots,N\} \Longrightarrow
                \sum_{\ell=1}^N \1_{\{ t_\ell\leq t\}}x_\ell K(t-t_\ell)\geq0,\quad \forall t\geq0.
	     \end{equation}
Here we use the convention $\sum_\varnothing \coloneqq0$. Moreover, We define the set $\mathcal{K}$ of non-increasing non-negativity preserving kernels and $\widehat{\mathcal{K}}\coloneqq\mathcal{K}\cup\mathcal{CM}$, where $\mathcal{CM}$ denotes the set of completely monotone kernels.
\end{definition}

 This condition was recently introduced and studied in \cite{alfonsi2023}, and allows us to show that the Volterra kernel preserves the order structure subject to the discretisation of the process via the splitting method. Additionally, we remark that every regular completely monotone kernel preserves non-negativity, see \cite[Theorem 2.11]{alfonsi2023}. Finally, to obtain the comparison principle for two different input data $(g_1,K,b_1, \sigma)$ and $(g_2, K, b_2, \sigma)$, we use the following definition.

\begin{definition}\label{def: comparable}
        Let $(g_i,K,b_i,\sigma)$ with $i=1,2$ satisfy Assumption \ref{assumption}, and suppose that $K\colon (0,\infty) \longrightarrow \R_+$ is continuous and $K\in\widehat{\mathcal{K}}$. We say that $(g_1,K,b_1,\sigma)$ and $(g_2,K,b_2,\sigma)$ are comparable, if the following conditions hold:
    \begin{enumerate}
        \item[(i)] There exist $h_1,h_2\in L^2_\loc(\R_+)$ and $\widetilde{g_1},\widetilde{g}_2\in C((0,\infty))$ such that
        \begin{align}\label{eq: g representation}
            g_i(t) = \widetilde{g}_i(t) + \int_0^t K(t-s)h_i(s)\, \d s, \qquad i = 1,2,
        \end{align}
        and $\widetilde{g}_2 - \widetilde{g}_1 \geq 0$ is non-decreasing on $(0,\infty)$.
        \item[(ii)] $b_1(t,x) + h_1(t)\leq b_2(t,x) + h_2(t)$ holds for each $t \geq 0$ and $x \in \R$.
    \end{enumerate}
\end{definition}

This class of comparable input data is sufficiently rich to cover the most frequently used examples from applications, as illustrated in Example \ref{example_frac_OU} and the discussion below. The following is our first result for non-negativity preserving kernels.

\begin{theorem} \label{comparison_cm_Lipschitz}
    Suppose that $(g_i,K,b_i,\sigma)$ with $i=1,2$ are comparable in the sense of Definition \ref{def: comparable}. Assume that there exists $\delta > 0$ with
    \begin{align}\label{eq: g h integrability}
     \int_0^T  \overline{g}_i(t)^{2+\delta} (1 + h_i(t)^2)\, \d t &< \infty, \qquad i = 1,2,
    \end{align}
    where $\overline{g}_i(t) = \sup_{s \in [t, T]}|g_i(s)|$, and that there exists $C > 0$ such that for a.a. $t\geq0$ and $x,y\in\R$,
    \begin{equation}\label{eq:lip}
    |b_1(t,x)-b_1(t,y)| + |b_2(t,x)-b_2(t,y)| + |\sigma(t,x)-\sigma(t,y)| \leq C|x-y|.
    \end{equation}
    Let $X^1, X^2$ be the unique strong solutions of \eqref{eq:sve} with input data $(g_i, K, b_i, \sigma)$, $i = 1,2$. Then $\P[X^1_t\leq X^2_t,\, t > 0] = 1$.
\end{theorem}
\begin{proof}
    First note that equation \eqref{eq:sve} has unique strong solutions $X^1, X^2$ with sample paths in $L_{\loc}^2(\R_+)$ due to Proposition \ref{prop: existence}. Moreover, given Assumption \ref{assumption}, it follows from Proposition \ref{prop_loc_hoelder} that $X^1 - g_1$, $X^2 - g_2$ have a modification with continuous sample paths. Hence $X^1,X^2$ are continuous on $(0,\infty)$. Finally, define
    \[
        \widetilde{b}_i(t,x) = b_i(t,x) + h_i(t),\quad t\in[0,T].
    \]
    Hence $\widetilde{b}_1 \leq \widetilde{b}_2$ and \eqref{eq:lip} still holds with the same $C$, while the linear growth constant becomes $|\widetilde{b}_1(t,x)| + |\widetilde{b}_2(t,x)| \leq \widetilde{C}(t)(1+|x|)$ with $\widetilde{C}(t) = C + |h_1(t)|+|h_2(t)|$, so that $\widetilde{C} \in L_{\loc}^2(\R_+)$. The proof is divided into three steps.

    \emph{Step 1:} We introduce a similar splitting method to Alfonsi \cite{alfonsi2023}, but now applied for non-constant possibly singular at zero $g$ and time-dependent coefficients. Suppose first that $K$ is nonnegative, nonincreasing and continuous on $\R_+$ with $K(0_+)>0$. Let $T>0$, $N\in\N$ be fixed and $t_k= kT/N,$ $k\in\{0,1,\dots,N\},$ be an equidistant grid of $[0,T]$ with step size $T/N.$ Recursively, we define the process $(\widehat{X}^i_t)_{t\in (0,T]}$ by setting
    \begin{align}
    \widehat{X}^i_t &= \widetilde{g}_i(t) + \sum_{k=1}^N \1_{[t_k,T]}(t)\left(K(t-t_k) \int_{t_{k-1}}^{t_k}[ \widetilde{b}_i(s,\xi^{i,k}_s)\,\d s + \sigma(s,\xi^{i,k}_s)\,\d B_s]\right)\nonumber
    \end{align}
    where $(\xi^{i,k}_t)_{t\in[t_{k-1},t_k)}$, $k\in\{1,\dots,N\},$ denotes the strong solution of the auxiliary SDE
    \begin{equation*}\label{eq:aux_sde}
    \xi^{i,k}_t = \widehat{X}_{t_k-}^i + \int_{t_{k-1}}^t K(0_+)[ \widetilde{b}_i(s,\xi^{i,k}_s)\,\d s + \sigma(s,\xi^{i,k}_s)\,\d B_s].
    \end{equation*}
    Under the given conditions, it follows from Proposition \ref{prop_approx_sing_g} that
    \begin{align}\label{eq: approximation X hat}
        \E[|\widehat{X}^i_t-X^i_t|^2]\longrightarrow 0, \qquad N\longrightarrow \infty.
    \end{align}
    Thus it suffices to prove that $\widehat{X}_t^1 \leq \widehat{X}_t^2$ a.s. for $t\in(0,T]$.

    \emph{Step 2:} We consider the case where $K$ is continuous and $K\in \mathcal{K}$. Let $T>0$ be fixed. For $N\in\N$, we define the corresponding approximations $\widehat{X}^1,\widehat{X}^2$ and denote by $\xi^1,\xi^2$ the auxiliary processes that arise in their construction. Note that we have by continuity of $K$ and $g$
    \begin{align*}
        \xi^{i,k}_{t_k-} &=  \widetilde{g}(t_k) + \sum_{\ell=1}^{k-1} K(t_k-t_\ell)\int_{t_{\ell-1}}^{t_\ell}[ \widetilde{b}_i(s,\xi^{i,\ell}_s)\,\d s + \sigma(s,\xi^{i,\ell}_s)\,\d B_s]\\ &\quad+ K(t_k-t_k)\int_{t_{k-1}}^{t_k} [ \widetilde{b}_i(s,\xi^{i,k}_s)\,\d s + \sigma(s,\xi^{i,k}_s)\,\d B_s] = \widehat{X}^i_{t_k}
    \end{align*}
    and similarly
    \begin{align*}
        &\widehat{X}^i_{t_k}-\widehat{X}^i_{t_k-}\\
        &\quad= \widetilde{g}_i(t_k) + \sum_{\ell=1}^k K(t_k-t_\ell) \int_{t_{\ell-1}}^{t_\ell}[ \widetilde{b}_i(s,\xi^{i,\ell}_s)\,\d s+ \sigma(s,\xi^{i,\ell}_s)\,\d B_s]
        \\ &\qquad - \widetilde{g}_i(t_k) -\sum_{\ell=1}^{k-1} K(t_k-t_\ell) \int_{t_{\ell-1}}^{t_\ell}[ \widetilde{b}_i(s,\xi^{i,\ell}_s)\,\d s + \sigma(s,\xi^{i,\ell}_s)\,\d B_s]\\
        &\quad= K(0_+)\int_{t_{k-1}}^{t_k}[ \widetilde{b}_i(s,\xi^{i,k}_s)\,\d s+ \sigma(s,\xi^{i,k}_s)\,\d B_s].
    \end{align*}
    Hence, we can represent them in compact form as
    \[
    \widehat{X}^i_t = \widetilde{g}_i(t) + \sum_{\ell=1}^N \1_{[t_\ell,T]}(t)\frac{\widehat{X}^i_{t_\ell}-\widehat{X}^i_{t_\ell-}}{K(0_+)}K(t-t_\ell),\qquad t\in(0,T].
    \]
    We show by induction on $k$ that $\widehat{X}^1_t\leq \widehat{X}^2_t$ on $(0,t_k]$.
	For $k=1,$ we have per assumption $\widehat{X}^1_t = \widetilde{g}_1(t) \leq \widetilde{g}_2(t) = \widehat{X}^2_t$ for $t\in (0,t_1)$. Using the Comparison Principle for SDEs, see \cite[Chapter  VI, Theorem 1.1]{ikeda_watanabe}, we obtain $\widehat{X}^1_{t_1} = \xi_{t_1-}^{1,1} \leq \xi_{t_1-}^{2,1} = \widehat{X}_{t_1}^2$. Suppose that the result is true for some $k\geq1$. Then, we have for all $j\in\{1,\dots,k\}$, by construction
	\begin{equation*}
			0 \leq \widehat{X}_{t_j}^2-\widehat{X}_{t_j}^1 = (\widetilde{g}_2-\widetilde{g}_1)(t_j)+ \sum_{\ell=1}^j \frac{(\widehat{X}_{t_\ell}^2-\widehat{X}_{t_\ell-}^2)-(\widehat{X}_{t_\ell}^1-\widehat{X}_{t_\ell-}^1)}{K(0_+)}K(t_j-t_\ell).
	\end{equation*}
	Since $K$ preserves non-negativity and $\widetilde{g}_2-\widetilde{g}_1\geq0$ is non-decreasing, we obtain by Lemma \ref{prop_non-neg_preserv} that $\widehat{X}_t^2-\widehat{X}_t^1\geq0$ holds for $t\in[t_k,t_{k+1})$. Using again the comparison result for SDEs yields
	\[
		\widehat{X}_{t_{k+1}}^1 = \xi^{1,k+1}_{t_{k+1}-}\leq \xi^{2,k+1}_{t_{k+1}-} = \widehat{X}^2_{t_{k+1}}.
	\]
	By induction, we get for $k=N-1$ that $\P[\widehat{X}^1_T\leq\widehat{X}^2_T]=1$. By \eqref{eq: approximation X hat} we conclude $\P[X^1_T\leq X^2_T]=1$. Since $T > 0$ was arbitrary and $X$ has continuous sample paths on $(0,\infty)$, we obtain the desired comparison principle.

    \emph{Step 3:} Let us now prove the assertion when $K\in\mathcal{CM}$ is completely montone. Using the Bernstein theorem, see \cite[Theorem 4.8]{SchillingSongVondracek+2012}, we find a Borel measure $\mu$ on $\R_+$ such that
	$K(t) = \int_{\R_+}\e{-\rho t}\,\mu(\d\rho)$ for $t\in(0,\infty)$. Define, for $H>0$, the approximation $K^H\colon\R_+\longrightarrow \R_+$ by $K^H(t) \coloneqq \int_{[0,H]}\e{-\rho t}\,\mu(\d\rho)$. Then $0 \leq K^H(t) \leq K^H(0)$ and $|(K^H)'(t)| \leq HK^H(0)$.  Hence, it satisfies Assumption \ref{assumption} with $\gamma=1/(2+\eta)$ and therefore
	\begin{equation*}\label{eq:volterra_kernel_kh}
			X^{i,H}_t = g_i(t) + \int_0^t K^H(t-s)b_i(s,X^{i,H}_s)\,\d s + \int_0^t K^H(t-s) \sigma(s,X^{i,H}_s)\,\d B_s
	\end{equation*}
    admits a unique, strong continuous solution for $i = 1,2$. Since regular completely monotone kernels are non-negativity preserving and satisfy \eqref{eq:pos_preserving}, see \cite[Theorem 2.11]{alfonsi2023}, we may apply steps 1 and 2 to obtain $\P[X^{1,H}_T\leq X^{2,H}_T]=1$ for each $T > 0$. Next, we prove that
    \[
        \lim_{H\longrightarrow \infty}\E[|X^{i}_T-X^{i,H}_T|^2]=0
    \]
    which completes the proof. Applying Lemma \ref{lemma kernel approx} yields
    \begin{equation*}
        \E[|X_t - X^{i,H}_t|^2] \leq C\left( \int_0^T |K(s)-K^H(s)|^{2+\eta}\,\d s \right)^{\frac{2}{2 + \eta}}
    \end{equation*}
    where $C>0$ is some constant. Finally, the right-hand side can be bounded by
    \begin{align*}
        \int_0^T |K(s)-K^H(s)|^{2 + \eta}\,\d s
        &\leq \int_0^T \left( \int_{[H,\infty)}\e{-\rho s}\,\mu(\d\rho) \right)^{2+\eta}\, \d s
        \\ &\leq \int_0^T \e{-sH(2+\eta)/2} K(s/2)^{2 + \eta}\, \d s.
    \end{align*}
    By dominated convergence, the latter tends to zero as $H \longrightarrow  \infty$, which proves the assertion.
\end{proof}

Note that the case where $g_i(t) = x_i$ for $i = 1,2$ was already covered in \cite{alfonsi2023}. Since the fractional Riemann-Liouville kernel satisfies Assumption \ref{assumption} and is completely monotone, Theorem \ref{comparison_cm_Lipschitz} can be applied to this case.

\begin{example}\label{example_frac_OU}
 Let $X^1, X^2$ be fractional Ornstein-Uhlenbeck processes given as the unique strong solutions of
 \[
  X^i_t = x_i \frac{t^{\gamma - 1}}{\Gamma(\gamma)} + \int_0^t \frac{(t-s)^{\alpha-1}}{\Gamma(\alpha)}(b_i+\beta X^i_s)\,\d s + \sigma\int_0^t \frac{(t-s)^{\alpha-1}}{\Gamma(\alpha)}\,\d B_s,
 \]
 where $b_1 \leq b_2$, $\beta,\sigma\in\R$, $\alpha\in(1/2,1)$, $\gamma > 0$, and $x_1 \leq x_2$. If $\beta \geq 0$, then the comparison principle holds. If $\beta < 0$, then the comparison principle holds if and only if $\gamma \geq \alpha$.
\end{example}
\begin{proof}
 It follows from \cite{affine_volterra} that $X^i$ is explicitly given by
 \[
  X_t^i = x_i t^{\gamma - 1}E_{\alpha,\gamma}(\beta t^{\alpha}) + \int_0^t s^{\alpha - 1}E_{\alpha,\alpha}(\beta s^{\alpha})b_i \,\d s + \int_0^t (t-s)^{\alpha - 1}E_{\alpha,\alpha}(\beta(t-s)^{\alpha})\,\d B_s
 \]
 where $E_{\alpha,\gamma}(z) = \sum_{n=0}^{\infty}\frac{z^n}{\Gamma(\alpha n + \gamma)}$ denotes the two-parameter Mittag-Leffler function. Hence
 \[
  X_t^2 - X_t^1 = (x_2-x_1) t^{\gamma - 1}E_{\alpha,\gamma}(\beta t^{\alpha}) + \int_0^t s^{\alpha - 1}E_{\alpha,\alpha}(\beta s^{\alpha})(b_2 - b_1) \,\d s.
 \]
 Since $E_{\alpha,\alpha}(\beta t^{\alpha}) \geq 0$ for any $\beta,t$, the comparison principle is satisfied if and only if $E_{\alpha,\gamma}(\beta t^{\alpha}) \geq 0$ for all $t \geq 0$. If $\beta \geq 0$, then $E_{\alpha,\gamma}(\beta t^{\alpha}) \geq 0$ and nothing needs to be shown. Suppose that $\beta < 0$. If $\gamma \geq \alpha$, then $t\longmapsto E_{\alpha,\gamma}(-t^{\alpha})$ is completely monotone, see \cite[Section 4.9.2]{mainardi_book}, and hence the comparison principle holds. Conversely, suppose that $\gamma < \alpha$. Then $(0,\infty)\ni t\longmapsto t^{\gamma-1} E_{\alpha,\gamma}(-|\beta|t^\alpha)$ is integrable and its Laplace transform is given by $\int_0^\infty t^{\gamma-1} E_{\alpha,\gamma}(-|\beta|t^\alpha)\e{-st}\,\d s = \frac{s^{\alpha-\gamma}}{s^\alpha + |\beta|}$ for $s \geq 0$. Evaluating this in $s = 0$ gives
 \[
    \int_0^{\infty} t^{\gamma - 1}E_{\alpha,\gamma}(-|\beta|t^{\alpha})\,\d s = 0
 \]
 and hence $(0,\infty)\ni t\longmapsto E_{\alpha,\gamma}(-|\beta|t^\alpha)$ needs to attain negative values. Consequently, the comparison principle fails to hold.
\end{proof}

While the above proof is direct, one may also seek to apply Theorem \ref{comparison_cm_Lipschitz} instead. Indeed, consider the fractional Riemann-Liouville kernel $K(t) = \frac{t^{\alpha-1}}{\Gamma(\alpha)}$ with $\alpha \in (1/2,1)$ and define $g_i(t) = x_i \frac{t^{\gamma-1}}{\Gamma(\gamma)}$ where we additionally assume that $\gamma \geq \alpha$. Then $\widetilde{g}_i(t) = 0$ and $h_i(t) = x_i\frac{t^{(\gamma-\alpha)-1}}{\Gamma(\gamma-\alpha)}$. In particular, condition (ii) from Definition \ref{def: comparable} reduces to
\[
        (b_2 - b_1) + (x_2 - x_1)\left( \beta + \frac{t^{(\gamma-\alpha)-1}}{\Gamma(\gamma-\alpha)} \right) \geq 0.
\]
Hence, Theorem \ref{comparison_cm_Lipschitz} is applicable for $\beta \geq 0$, while for $\beta < 0$ it is generally not applicable.

\subsection{Non-Lipschitz case}

Next, let us examine cases where the coefficients $b_1,b_2, \sigma$ are not Lipschitz continuous. Since, in this generality, the weak existence of solutions is not guaranteed from the literature, below we complement the latter by proving the weak existence of continuous solutions when the function $g$ may be singular in $t=0$. Our proof below employs some tightness arguments. However, in contrast to existing arguments such as \cite{affine_volterra, proemel_scheffels_weak}, here we need to introduce an auxiliary weighted space of continuous functions to control the order of singularity of the function $g$.

\begin{theorem}\label{weakexistence_cm_lg}
    Let $K\colon (0,\infty) \longrightarrow \R_+$ be continuous and nonincreasing. If $(g,K,b,\sigma)$ satisfies Assumption \ref{assumption} and $g$ is continuous on $(0,\infty)$ such that there exists $\delta \in (0,1/2)$ such that for every $T>0$
    \begin{align}\label{eq: g growth condition}
            \sup_{t \in (0,T]}t^{\delta}|g(t)| < \infty,
    \end{align}
    then the stochastic Volterra equation \eqref{eq:sve} admits a continuous weak solution.
\end{theorem}
\begin{proof}
    Without loss of generality, we assume that $K$ is not identically zero since the existence of a continuous solution is otherwise trivial. Following \cite[Lemma 3.6]{proemel_scheffels_weak}, we define for $f\in\{b,\sigma\}$ a sequence of functions
    \begin{align}\label{eq: cm regularization}
    f_n(t,x) = \psi_n(x)\int_\R f(t,x-y)\varphi_n(y)\,\d y,\quad n\in\N,
    \end{align}
    where $\psi_n\colon\R\longrightarrow [0,1]$ with $\operatorname{supp}\psi_n\subset[-(n+1),n+1]$ and $\psi_n(x)=1$ for $x\in[-n,n]$, where $\varphi(y)=\frac{1}{c_n}(1-y^2)^n\1_{[-1,1]}(y)$ with $c_n=\int_{[-1,1]}(1-y^2)^n\,\d y$. Then $(f_n)_{n\in\N}$ is a sequence of Lipschitz continuous functions in the second variable and $\lim_{n\to \infty}f_n(t,\cdot)=f(t,\cdot)$ uniformly on compacts for all $t\geq0$. Moreover, the following growth condition is satisfied uniformly in $n$,
    \begin{equation}\label{eq:uni_lg}
    |b_{n}(t,x)|+|\sigma_n(t,x)|\leq C(1+|x|)^\xi,\quad \text{a.a. }t\geq0,x\in\R.
    \end{equation}
    Let $(X^n)_{n\in\N}$ be the family of processes constructed for $(g,K,b_{n},\sigma_n)$ on some fixed filtered probability space $(\Omega,\F,\mathbb{F},\P)$. Thanks to Proposition \ref{prop_loc_hoelder} and \eqref{eq:uni_lg},
    one readily checks that for $T>0$ and all $s,t\in[0,T]$,
    \begin{equation*}\label{eq:kolomogorov_estimate}
    \sup_{n\in\N}\E[{|(X_t^n - g(t))- (X_s^n - g(s))|^p}] \leq  c |t-s|^{1+\theta}
    \end{equation*}
    holds for some $\theta>0$, $p \geq 2$ and a constant $c$ that only depends on $C$, $T$, $K$, $p$. By applying Kolmogorov's tightness criterion (see e.g. \cite[Chapter XIII, Theorem 1.8]{revuz2004continuous}), we conclude that $(X^n - g)_{n\in\N}$ is a tight sequence of processes and hence has a weakly convergent subsequence on the path space $C(\R_+)$ with limit denoted by $\overline{\P}_g$. By abuse of notation, we denote this subsequence also by $(X^{n}-g)_{n\in\N}$. Let $C(\R_+)_0\subset C(\R_+)$ be the space of continuous functions that vanish at $t=0$. Since $\mathcal{L}(X^n-g)(C(\R_+)_0)=1$ it follows by the Portmanteau theorem that
    \[
    1=\limsup_{n\to\infty}\mathcal{L}(X^n-g)(C(\R)_0)\leq \overline{\P}_g(C(\R_+)_0).
    \]
    Next, to prove weak convergence of $(X^n)_{n \in \N}$, let $C_{\kappa}([0,T])$ with $\kappa > 0$ be the Banach space of continuous functions $f\colon(0,T]\longrightarrow \R$ with finite norm $\norm{f}_{\kappa,T}\coloneqq \sup_{t\in(0,T]}t^{\kappa}|f(t)|$ such that $\lim_{t\downarrow0} t^{\kappa} f(t)$ exists. Note that $\iota\colon C_{\kappa}([0,T])\longrightarrow  C([0,T])$ defined by $\iota(f)(t) = t^\kappa f(t)$ satisfies $\norm{\iota(f)}_{[0,T]} = \norm{f}_{\kappa,T}$, and hence is an isometric isomorphism, whence $C_{\kappa}([0,T])$ is a polish space. Finally, we let $C_\kappa(\R_+)\coloneqq \bigcap_{k\in\N}C_\kappa([0,k])$ be equipped with the metric
    \[
    \varrho(f,g) \coloneqq \sum_{k=1}^\infty 2^{-k}(\norm{f-g}_{\kappa,k}\land 1).
    \]
    By assumption \eqref{eq: g growth condition}, $g \in C_{\kappa}(\R_+)$ for $\kappa > \delta$. Hence $\Phi\colon C(\R_+)_0\longrightarrow C_\kappa(\R_+)$ given by $\Phi(f) = f+g$ defines a continuous mapping. Thus, $(\Phi(X^n-g))_{n\in\N}=(X^n)_{n\in\N}$ is weakly convergent on $C_\kappa(\R_+)$ with weak limit denoted by $\overline{\P}$. Note that $\overline{\P}_g, \overline{\P}$ are related by
    \begin{align}\label{eq: Pg P relation}
     \overline{\P}_g \circ \Phi^{-1} = \overline{\P}.
    \end{align}
    It remains to show that $\overline{\P}$ determines a weak solution of \eqref{eq:sve}.

    Let $L$ be the resolvent of the first kind of $K$\footnote{i.e. a locally finite (positive) measure $L$ such that $L \ast K = K\ast L = 1$. The latter exists since $K\neq0$ is non-negative and non-increasing, see \cite[Theorem 5.5.5]{grippenberg}.}. Let $(M_t^n)_{t \geq 0}$ be the continuous $\P$-martingale for the filtration $\mathbb{F}$ given by
    \[
     M_t^n = \int_0^t \sigma_n(s,X^n_s) \,\d B_s = \int_{[0,t]}(X^n_{t-s} - g(t-s))\, L(\d s) - \int_0^t b_{n}(s,X^n_s)\,\d s
    \]
    where the second equality follows from \cite[Lemma 2.6]{affine_volterra}. Moreover, since \eqref{eq:uni_lg} holds uniformly in $n\in\N$, it follows that $(M_t^n)_{n\in\N}$ is uniformly integrable due to
    \begin{align*}
         \sup_{n\in\N}\E\left[ \sup_{t \in [0,T]}|M_t^n|^2 \right]
         &\leq \sup_{n\in\N}\int_0^T \E\left[\sigma_n(s,X_s^n)^2 \right]\d s
         \\ &\lesssim C \left(T + \sup_{n\in\N}\int_0^T\E\left[ |X_s^n|^2 \right]\d s \right) < \infty.
    \end{align*}

    Define a continuous mapping $\mathcal{M} \colon C(\R_+)_0 \longrightarrow C(\R_+)_0$ by
    \begin{equation*}
        \mathcal{M}_t(w) = \int_{[0,t]}w(t-s)\, L(\d s) - \int_0^t b(s,\Phi(w)(s))\, \d s,
    \end{equation*}
    where the continuity is guaranteed by  \cite[Corollary 3.6.2 (iii)]{grippenberg}.
    Since $X^n - g \in C(\R_+)_0$, by continuous mapping theorem combined with \eqref{eq: Pg P relation}, we obtain \begin{align*}\label{eq: weak convergence martingale}
        M^n = \mathcal{M}(X^n-g) \Longrightarrow \overline{\P}_g \circ \mathcal{M}^{-1} = \overline{\P} \circ (\mathcal{M}^{\Phi})^{-1}
    \end{align*}
    where $\mathcal{M}^{\Phi} \colon\Phi(C(\R_+)_0)\longrightarrow C(\R_+)$ is given by
    \[
     \mathcal{M}_t^{\Phi}(w) \coloneqq \mathcal{M}_t(\Phi^{-1}(w)) = \int_{[0,t]}(w(t-s)- g(t-s))\, L(\d s) - \int_0^t b(s,w(s))\, \d s.
    \]
    Next, we show that $(\mathcal{M}^{\Phi}_t)_{t \geq 0}$ is a continuous $\overline{\P}$-martingale for the natural filtration generated by coordinate process on $C_{\kappa}(\R_+)$, and determine its quadratic variation. The continuity follows from $\Phi(C(\R_+)_0) = \{f\in C_\kappa(\R_+) : f-g\in C(\R_+)_0\}$, \cite[Corollary 3.6.2 (iii)]{grippenberg} and $\overline{\P}(\Phi(C(\R_+)_0)) = \overline{\P}_g(C(\R_+)_0)=1.$
    Let $s<t$, $m\in\N$, $f\colon \R^{m}\longrightarrow \R$ bounded and continuous, and $0\leq s_1\leq\dots\leq s_m\leq s$. Let $(x(t))_{t\geq 0}$ denote the coordinate process $x(t)\colon C_{\kappa}(\R_+; \R) \longrightarrow \R, \ \ w \longmapsto w(t)$. Applying \cite[Theorem 3.5]{billingsley} yields
    \begin{equation}\label{eq:mon_class}
     \E^{\overline{\P}}\left[f(x(s_1),\dots,x(s_m))\left(\mathcal{M}_t^{\Phi} - \mathcal{M}_s^{\Phi}\right) \right]
     = \lim_{n\to \infty}\E[f(X^{n}_{s_1},\dots,X^{n}_{s_m})(M_t^n - M_s^n)] = 0.
    \end{equation}
    A monotone class argument implies that $(\mathcal{M}^{\Phi}_t)_{t \geq 0}$ is a $\overline{\P}$-martingale. To determine its quadratic variation, let us write
     \begin{align*}
     \langle M^n \rangle = \int_0^{\cdot} \left(\sigma_n(s,X_s^n)^2 - \sigma(s,X_s^n)^2\right)\, \d s + \int_0^{\cdot} \sigma(s,X_s^n)^2\, \d s.
     \end{align*}
     Since $X^n \Longrightarrow \overline{\P}$, the continuous mapping theorem implies that
     \[
        \int_0^{\cdot}\sigma(s,X_s^n)^2\, \d s \Longrightarrow \mathcal{L}_{\overline{\P}}\left( \int_0^{\cdot}\sigma(s, x(s))\, \d s \right).
     \]
     Thus, by Slutsky's theorem, it suffices to show that the first term converges to zero in probability. Let $T > 0$, then
    \begin{align*}
        &\ \int_0^T \E\left[ \left|\sigma_n(s,X_s^n)^2 - \sigma(s,X_s^n)^2 \right| \right] \d s
        \\ &\lesssim \left(\int_0^T \E\left[ \sigma_n(s,X_s^n)^2 + \sigma(s,X_s^n)^2 \right] \d s \right)^{1/2} \left(\int_0^T \E\left[ \left( \sigma_n(s,X_s^n) - \sigma(s,X_s^n)\right)^2 \right] \d s \right)^{1/2}
        \\ &\lesssim \left(1 + \sup_{n \in\N}\int_0^T \E[|X_s^n|^2]\, \d s \right)^{1/2} \left(\int_0^T \E\left[ (\sigma_n(s,X_s^n) - \sigma(s,X^n_s))^2\right] \d s \right)^{1/2}
        \\ &\lesssim R^{-\lambda/2}\left(1 + \sup_{n \in\N}\int_0^T \E[|X_s^n|^2] \d s \right)^{1/2} \left( 1 + \sup_{n \in\N}\int_0^T \E\left[ |X_s^n|^{2+\lambda} \right]  \d s \right)^{1/2}
        \\ &\qquad + \left(1 + \sup_{n \in\N}\int_0^T \E[|X_s^n|^2]\, \d s \right)^{1/2} \left(\int_0^T \norm{\sigma_n(s,\cdot)-\sigma(s,\cdot)}_{[-R,R]}^2 \, \d s \right)^{1/2}
    \end{align*}
    where we have used \eqref{eq:uni_lg} and the inequality
    \begin{align*}
    \E[|\sigma_n(s,X^n_s)-\sigma(s,X_s^n)|^2]
    &\lesssim \E\left[\1_{\{|X^n_s|>R\}}\left(|\sigma_n(s,X^n_s)|^2+|\sigma(s,X_s^n)|^2 \right)\right]
    \\ &\qquad \qquad + \norm{\sigma_n(s,\cdot)-\sigma(s,\cdot)}_{[-R,R]}^2
    \\ &\lesssim \frac{C}{R^\lambda}\left( 1 + \E[|X^n_s|^{2+\lambda}]\right)+ \norm{\sigma_n(s,\cdot)-\sigma(s,\cdot)}_{[-R,R]}^2,
    \end{align*}
    where $\lambda= q-2>0$ if $g\in L^{q}_\loc(\R_+), q\in(2,\infty),$ and $\lambda=1$ if $g\in L^\infty_\loc(\R_+)$.
    Thus, letting first $n \longrightarrow  \infty$ and then $R \longrightarrow  \infty$, proves the assertion. Arguing as in \eqref{eq:mon_class} yields that $(\mathcal{M}^\Phi)^2-\int_0^\cdot \sigma(s,x(s))^2\,\d s$ is an $\overline{\P}$-martingale. Consequently, we have shown that $(\mathcal{M}^{\Phi}_t)_{t\geq 0}$ is a continuous $\overline{\P}$-martingale with quadratic variation
    \[
    \langle \mathcal{M}^{\Phi}\rangle_t = \int_0^t \sigma(s,x(s))^2 \,\d s.
    \]
    By martingale representation theorem \cite[Chapter V, Theorem 3.9]{revuz2004continuous}, there exists, possibly on an enlargement of the probability space, a Brownian motion $W$ and a predictable process $\Sigma$ such that $\Sigma_s^2 = \sigma(s,\widehat{X}_s)^2$ and $\mathcal{M}^{\Phi}_t = \int_0^t \Sigma_s\,\d W_s$ where $\widehat{X}$ denotes the processes $x$ considered on the enlargement such that $\mathcal{L}(\widehat{X})=\mathcal{L}(x)$. Consequently, by \cite[Lemma 2.6]{affine_volterra}
    \begin{align*}
     \widehat{X}_t &= g(t) + \int_0^t K(t-s) b(s,\widehat{X}_s)\, \d s + \int_0^t K(t-s) \sigma(s,\widehat{X}_s)\,\d W_s
    \end{align*}
    which completes the proof.
\end{proof}

Remark that $g$ given as in \eqref{eq: g growth condition} automatically satisfies $g \in L^{q}_{\loc}(\R_+)$, $q\in(0,1/\delta)$, and hence condition \eqref{eq: prime condition} is satisfied whenever
\[
    2\xi \delta < \frac{\gamma + \frac{1}{2}\frac{\eta}{2+\eta}}{1 + \frac{1}{\eta}}.
\]

\begin{remark}\label{remark:boundary approxy}
    If $b(t,0)\geq0$ and $\sigma(t,0)=0$ for $t\geq0$, we can always construct approximating sequences $(\widetilde{b}_n)_{n\in\mathbb{N}}$ and $(\widetilde{\sigma}_n)_{n\in\mathbb{N}}$ that preserve these boundary conditions. Specifically, there exist sequences retaining the Lipschitz continuity, the uniform sublinear growth, and the convergence properties stated in the preceding proof, while additionally satisfying
    \begin{equation}\label{eq: b sigma pos}
    \widetilde{b}_n(t,0)\geq0 \quad\text{and}\quad \widetilde{\sigma}_n(t,0)=0,\quad t\geq0, n\in\mathbb{N}.
    \end{equation}
\end{remark}
\begin{proof}
    We define $\widetilde{b}_n$ and $\widetilde{\sigma}_n$ by
    \[
    \widetilde{b}_n(t,x)\coloneqq b_n(t,x)+\max\{0,-b_n(t,0)\} \quad\text{and}\quad \widetilde{\sigma}_n(t,x)\coloneqq \sigma_n(t,x)-\sigma_n(t,0),
    \]
    where $b_n$ and $\sigma_n$ are defined by the regularization \eqref{eq: cm regularization}. It is immediate that $\widetilde{b}_n$ and $\widetilde{\sigma}_n$ satisfy the boundary conditions \eqref{eq: b sigma pos}. Furthermore, because the adjustment terms depend only on $t$, both sequences inherit the uniform Lipschitz continuity in the second variable.
    They also preserve the uniform sublinear growth condition. Indeed, demonstrating this for $\widetilde{b}_n$, we have
    \begin{multline*}
    |\widetilde{b}_n(t,x)|\leq |b_n(t,x)|+\max\{0,-b_n(t,0)\}\leq C(1+|x|)^\xi + |b_n(t,0)|\\ \leq C(1+|x|)^\xi+ C(1+|0|)^\xi \leq 2C(1+|x|)^\xi.
    \end{multline*}
    An identical argument applies to $\widetilde{\sigma}_n$. Finally, thanks to $b(t,0)\geq0$ and $\sigma(t,0)=0$, combined with the pointwise convergence of the original sequences at $x=0$, the adjustment terms vanish as $n \to \infty$.
\end{proof}

As another consequence, we may also prove the existence of nonnegative continuous weak solutions. The latter complements the results obtained in \cite{markovian_structure}.\\

\begin{corollary}\label{weakexistence_cm_lg_nonnegative}
    Let $K\colon (0,\infty) \longrightarrow \R_+$ be continuous with $K\in\widehat{\mathcal{K}}$,  and let $(g,K,b,\sigma)$ satisfy Assumption \ref{assumption} and additionally suppose that
    \[
    b(t,0)\geq0 \quad\text{and}\quad \sigma(t,0)=0.
    \]
    If $g = \widetilde{g} + K \ast h$ is such that $0 \leq \widetilde{g} \in C((0,\infty)) \cap L_{\mathrm{loc}}^2(\R_+)$ is nondecreasing, $0 \leq h \in L_{\mathrm{loc}}^2(\R_+)$, \eqref{eq: g h integrability},\eqref{eq: g growth condition} hold for some $\delta \in (0,1/2)$, then the stochastic Volterra equation \eqref{eq:sve} admits a nonnegative continuous weak solution.
\end{corollary}
\begin{proof}
    Let $X^n$ be the unique strong solution of \eqref{eq:sve} with $b,\sigma$ replaced by $\widetilde{b}_n, \widetilde{\sigma}_n$, i.e.
    \[
        X_t^n = g(t) + \int_0^t K(t-s)\widetilde{b}_n(s,X_s^n)\, \mathrm{d}s + \int_0^t K(t-s)\widetilde{\sigma}_n(s, X_s^n)\, \mathrm{d}B_s.
    \]
    where $\widetilde{b}_n,\widetilde{\sigma}_n$ are defined as in Remark \ref{remark:boundary approxy}.
    By the convergence arguments given in the proof of Theorem \ref{weakexistence_cm_lg}, it suffices to prove that $\P[X^n_t \geq 0,\,t>0]=1$. To prove the latter, we apply Theorem \ref{comparison_cm_Lipschitz} to the pair of processes $X^n$ and $Y \equiv 0$. Indeed, note that $Y \equiv 0$ is the unique solution of
    \[
    Y_t^n = \widehat{g}_n(t) + \int_0^t K(t-s)\widetilde{b}_n(s,Y^n_s) \,\d s + \int_0^t K(t-s)\widetilde{\sigma}_n(s,Y^n_s)\,\d B_s
    \]
    where $\widehat{g}_n(t) = -\int_0^t K(t-s)\widetilde{b}_n(s,0)\, \mathrm{d}s$. It is easy to check that $(g,K,\widetilde{b}_n, \widetilde{\sigma}_n)$ and $(\widehat{g}_n, K, \widetilde{b}_n, \widetilde{\sigma}_n)$ are comparable in the sense of Definition \ref{def: comparable}. Hence Theorem \ref{comparison_cm_Lipschitz} gives $\P[X^n \geq Y = 0,\,t>0]=1$, which proves the assertion.
\end{proof}

Below, we illustrate this existence result by an application to a general one-dimensional nonnegative process with a power-law diffusion coefficient. The latter contains the rough Cox-Ingersoll-Ross process as a particular case.

\begin{example}\label{example: rough VCIR}
    Let $\lambda, \theta, \sigma > 0$, $\gamma_1,\gamma_2 \in (0,1]$, and $\alpha \in (1/2,1)$. Take $g$ of the form
    \[
        g(t) = \widetilde{g}(t) + \int_0^t \frac{(t-s)^{\alpha-1}}{\Gamma(\alpha)} h(s)\, \mathrm{d}s
    \]
    such that $0 \leq \widetilde{g} \in C((0,\infty)) \cap L_{\mathrm{loc}}^2(\R_+)$ is nondecreasing, $0 \leq h \in L_{\mathrm{loc}}^2(\R_+)$, and \eqref{eq: g h integrability}, \eqref{eq: g growth condition} hold for some $\delta \in (0,1/2)$. If Assumption \ref{assumption} is satisfied for
    \[
        X_t = g(t) + \int_0^t \frac{(t-s)^{\alpha-1}}{\Gamma(\alpha)}\lambda\left( \theta - X_s\right)^{\gamma_1}\, \mathrm{d}s + \sigma \int_0^t \frac{(t-s)^{\alpha-1}}{\Gamma(\alpha)},X_s^{\gamma_2}\, \mathrm{d}B_s,
    \]
    then it admits a continuous nonnegative weak solution.
\end{example}

Next, we turn to the comparison principle beyond Lipschitz continuous coefficients. Since a comparison principle as formulated in Theorem \ref{comparison_cm_Lipschitz} would imply pathwise uniqueness of \eqref{eq:sve}, it is in this generality out of reach. However, by approximation, we show that there always exists a couple of solutions $(X^1, X^2)$ defined on a joint filtered probability space that satisfies the comparison principle, i.e., the existence of a monotone coupling.

\begin{theorem}\label{comparison_cm_lg}
    Let $(g_i,K,b_i,\sigma)$, $i=1,2,$ be comparable in the sense of Definition \ref{def: comparable}. Suppose there exists $\delta \in (0,1/2)$ with \eqref{eq: g growth condition} holds for $g_1,g_2$, and that $g_1,g_2$ satisfy \eqref{eq: g h integrability}. Then there exists a continuous weak solution $((\widehat{X}^1,\widehat{X}^2),(\widehat{\Omega},\widehat{\F},\widehat{\mathbb{F}},\widehat{\P}),\widehat{B})$ of the two dimensional stochastic Volterra equation
       \begin{align}\label{eq: joint coupling1}
       \begin{pmatrix}
           \widehat{X}^1_t\\\widehat{X}^2_t
       \end{pmatrix}=
       \begin{pmatrix}
           g_1(t)\\g_2(t)
       \end{pmatrix}+\int_0^t K(t-s)
       \begin{pmatrix}
           b_1(s,\widehat{X}^1_s)\\b_2(s,\widehat{X}^2_s)
       \end{pmatrix}\d s +\int_0^t K(t-s)
       \begin{pmatrix}
           \sigma(s,\widehat{X}^1_s)\\ \sigma(s,\widehat{X}^2_s)
       \end{pmatrix}\d\widehat{B}_s
       \end{align}
       such that $\widehat{\P}[\widehat{X}^1_t\leq \widehat{X}^2_t,\,t>0]=1$. In particular, its marginals are continuous weak solutions of \eqref{eq:sve}.
\end{theorem}
\begin{proof}
Let us define $b_{1,n}, b_{2,n}, \sigma_n$ as in \eqref{eq: cm regularization}. Since $b_1 \leq b_2$, it is clear that this approximation inherits the ordering condition $b_{1,n}\leq b_{2,n}$. Let $(X^n)_{n\in\N}, (Y^n)_{n \in \N}$ be the family of processes constructed for $(g_i,K,b_{i,n},\sigma_n)$ on some fixed filtered probability space $(\Omega,\F,\mathbb{F},\P)$. By Theorem \ref{comparison_cm_Lipschitz}, $\P[X^n_t \leq Y^n_t,t > 0] = 1$.

Following the arguments presented in the proof of Theorem \ref{weakexistence_cm_lg}, we firstly deduce that $((X^n - g_1,Y^n-g_2))_{n\in\N}$ is a tight sequence of processes and hence has a weakly convergent subsequence on the path space $C(\R_+;\R^2)$. Secondly, we extend the spaces $C(\R_+)_0\coloneqq C(\R_+;\R)_0$, $C_\kappa(\R_+)\coloneqq C_\kappa(\R_+;\R)$ in an obvious way to $C(\R_+;\R^2)_0$, $C_\kappa(\R_+;\R^2)$ with $\kappa>\delta$ and infer that $((X^n,Y^n))_{n\in\N}$ is weakly convergent on $C_\kappa(\R_+;\R^2)$ with weak limit denoted by $\overline{\P}.$ To prove the desired comparison result, it is essential to observe that $\overline{\P}$ has inherited the ordering structure. Indeed, let $(x,y)=((x,y)(t))_{t\geq0}$ be the coordinate process
\[
   (x,y)(t)\colon C_\kappa(\R_+;\R^2)\longrightarrow  \R^2,\quad w=(w_1,w_2)\longmapsto(w_1(t),w_2(t)).
 \]
 Let $A\coloneqq\{w=(w_1,w_2)\in C_\kappa(\R_+;\R^2)\mid w_2(t)-w_1(t)\geq0,t>0\}$. Writing
 \[
    A_t\coloneqq\{(w_1,w_2)\in C_\kappa(\R_+;\R^2)\mid w_2(t)-w_1(t)\geq0\}=\pi_t^{-1}([0,\infty))
 \]
 where $\pi_t\colon (C_\kappa(\R_+;\R^2),\varrho)\longrightarrow \R$, $\pi_t(w_1,w_2)=w_2(t)-w_1(t)$ is a continuous function, we conclude that $A_t$ is closed in $((C_\kappa(\R_+;\R^2),\varrho)$. Hence also $A = \bigcap_{t>0}A_t$ is a closed subset of $C_\kappa(\R_+;\R^2)$. An application of the Portmanteau theorem yields
 \begin{equation*}
         \overline{\P}[x(t)\leq y(t),t\geq0] = \overline{\P}[A] \geq \limsup_{n\to \infty}\P_{n}[A] = \limsup_{n\to \infty}\P[X_t^{n}\leq Y_t^{n}, t > 0]=1.
 \end{equation*}
  Finally, we show that $\overline{\P}$ determines a weak solution to the 2-dimensional stochastic Volterra equation \eqref{eq: joint coupling1}. Similarly to the proof of Theorem \ref{weakexistence_cm_lg}, we define the continuous $\P$-martingales
    \[ M^n_t\coloneqq
       \int_0^t \begin{pmatrix} \sigma_n(s,X^n_s) \\ \sigma_n(s,Y^n_s) \end{pmatrix} \,\d B_s
    = \begin{pmatrix}
        \int_{[0,t]}X^n_{t-s}-g_1(t-s)\,L(\d s) -\int_0^t b_{1,n}(s,X^n_s)\,\d s\\ \int_{[0,t]}Y^n_{t-s}-g_2(t-s)\,L(\d s) -\int_0^t b_{2,n}(s,X^n_s)\,\d s
    \end{pmatrix},
    \]
    where $L$ denotes the resolvent of the first kind of  $K$. Moreover, we observe that $M^n\Longrightarrow \overline{\P}\circ\mathcal{M}^{-1}$, where
    $\mathcal{M}$ is on $\{(w_1,w_2)\in C_\kappa(\R_+;\R^2): (w_1-g_1,w_2-g_2)\in C(\R_+;\R^2)_0\}$ given by
    \begin{equation}\label{eq:curly_M}
    \mathcal{M}_t(w_1,w_2) \coloneqq \begin{pmatrix}
        \int_{[0,t)}w_1(t-s)-g_1(t-s)\,L(\d s) - \int_0^t b_{1}(s,w_1(s))\,\d s\\ \int_{[0,t)}w_2(t-s)-g_2(t-s)\,L(\d s) - \int_0^t b_{2}(s,w_2(s))\,\d s\\
    \end{pmatrix}
    \end{equation}
    Furthermore, we can show that $(\mathcal{M}_t)_{t\geq0}$ is a continuous $\overline{\P}$-martingale for the natural filtration generated by the coordinate process on $C_\kappa(\R_+;\R^2)$ and its quadratic variation is given by
    \[
    \langle \mathcal{M}  \rangle_t = \int_0^t \begin{pmatrix}
         \sigma(s,x(s))^2 & \sigma(s,x(s))\sigma(s,y(s))\\ \sigma(s,x(s))\sigma(s,y(s)) & \sigma(s,y(s))^2
     \end{pmatrix} \d s.
    \]

    By martingale representation theorem \cite[Chapter V, Theorem 3.9]{revuz2004continuous}, there exists, possibly on an enlargement of the probability space, a two-dimensional Brownian motion $(W^1,W^2)$ and a $\R^{2\times2}$-valued predictable process $\Sigma$ such that
    \[
    \mathcal{M}(t) = \int_0^t \Sigma_s\,\d\!\begin{pmatrix}
        W^1_s\\W^2_s
    \end{pmatrix} \qquad \text{ with } \qquad \Sigma_s \Sigma_s^{\top} = \begin{pmatrix}
         \sigma(s,\widehat{X}_s)^2 & \sigma(s,\widehat{X}_s)\sigma(s,\widehat{Y}_s)\\ \sigma(s,\widehat{X}_s)\sigma(s,\widehat{Y}_s) & \sigma(s,\widehat{Y}_s)^2
     \end{pmatrix}
    \]
    where $\widehat{X},\widehat{Y}$ denote the processes $x,y$ considered on the enlargement such that $\law(\widehat{X})=\law(x)$, $\law(\widehat{Y})=\law(y)$.
    Note that any other choice of $\widetilde{\Sigma}$ that satisfies $\widetilde{\Sigma}\widetilde{\Sigma}^{\top} = \Sigma \Sigma^{\top}$ gives the same law of $\mathcal{M}$ under $\overline{\mathbb{P}}$. Let us define
    \[
    \widehat{\Sigma}_t = \frac{1}{\sqrt{2}} \begin{pmatrix}
        \sigma(s,\widehat{X}_t) & \sigma(s,\widehat{X}_t)\\ \sigma(s,\widehat{Y}_t) & \sigma(s,\widehat{Y}_t)
    \end{pmatrix} \ \text{ and } \ \widehat{\mathcal{M}}(t) = \int_0^t \widehat{\Sigma}_s \,\d\!\begin{pmatrix}
        W^1_s\\ W^2_s
    \end{pmatrix}.
    \]
    Then $\langle \widehat{\mathcal{M}}\rangle_t = \int_0^t \widehat{\Sigma}_s \widehat{\Sigma}_s^{\top}\,\d s = \int_0^t \Sigma_s \Sigma_s^{\top}\,\d s = \langle \mathcal{M}\rangle_t$. Hence $\widehat{\mathcal{M}}$ and $\mathcal{M}$ have the same law. Consequently, by \eqref{eq:curly_M} and \cite[Lemma 2.6]{affine_volterra},
    \begin{align*}
    \begin{pmatrix}
        \widehat{X}_t\\ \widehat{Y}_t
    \end{pmatrix} &=
    \begin{pmatrix}
        g_1(t)\\g_2(t)
    \end{pmatrix}+
    \int_0^t K(t-s)
    \begin{pmatrix}
        b_1(s,\widehat{X}_s)\\b_2(s,\widehat{Y}_s)
    \end{pmatrix}\d s+\frac{1}{\sqrt{2}}
    \int_0^t K(t-s)
    \begin{pmatrix}
        \sigma(s,\widehat{X}_s)\,\d(W^1_s + W^2_s)\\ \sigma(s,\widehat{Y}_s)\,\d(W^1_s +W^2_s)
    \end{pmatrix}\\
    &=\begin{pmatrix}
        g_1(t)\\g_2(t)
    \end{pmatrix}+
    \int_0^t K(t-s)
    \begin{pmatrix}
        b_1(s,\widehat{X}_s)\\b_2(s,\widehat{Y}_s)
    \end{pmatrix}\d s+
    \int_0^t K(t-s)
    \begin{pmatrix}
        \sigma(s,\widehat{X}_s)\\ \sigma(s,\widehat{Y}_s)
    \end{pmatrix}\d\widehat{B}_s ,
    \end{align*}
     where $\widehat{B}=\frac{1}{\sqrt{2}}(W^1+W^2)$ is a one-dimensional Brownian motion. This completes the proof.

\end{proof}

Also here, let us note that the model given in Example \ref{example: rough VCIR} satisfies the conditions of Theorem \ref{comparison_cm_lg}, and hence fulfils the comparison principle.

\section{Regular Case}\label{section: Regular Case}

\subsection{Characterisation of the comparison principle}

In this section, we consider the case of regular Volterra kernels. Hence, the corresponding Volterra process is a semimartingale. Note that, while for completely monotone Volterra kernels (including $K = 1$), the comparison principle holds, for general Volterra kernels this does not need to be the case. As a first step, we provide a general characterisation of the comparison principle for sufficiently regular drift $b$ and diffusion coefficient $\sigma$.

\begin{proposition}\label{prop:Comparion control}
    Let $b, \sigma: [0,T] \times \R \longrightarrow \R$ be measurable such that
    \begin{itemize}
        \item[(i)] $b(t,x)$ is globally Lipschitz continuous in $x$ and there exist $c>0$, $\kappa\in[0,1)$ such that
        \[
        |b(t,x)|\leq c(1+|x|^\kappa),\quad t\in[0,T],x\in\R;
        \]
        \item[(ii)] $\sigma(t,x)$ is once continuously partially differentiable in $t$ and twice in $x$. Moreover, $\sigma,\partial_x\sigma$ are bounded, globally Lipschitz continuous in $x$ and $1/2$-H\"older continuous in $t$. Finally, for $\partial_t \sigma,\partial_{xx}\sigma$ there exist $c,\eta>0$ such that
        \[
        |\partial_t \sigma(t,x)|+|\partial_{xx}\sigma(t,x)|\leq c(1+|x|^\eta),\quad t\in[0,T],x\in\R.
        \]
    \end{itemize}
    Let $K \in C^1([0,T])$ and for $x \in \R$ let $X^x$ be the unique strong solution of
    \[
        X^x_t = x + \int_0^t K(t-s)b(s, X_s^x) \, \mathrm{d}s + \int_0^t K(t-s)\sigma(s, X_s^x)\, \mathrm{d}B_s,\quad  t\in[0,T].
    \]
    Then the following are equivalent:
    \begin{enumerate}
        \item[(a)] (Comparison property) For any $x\leq y$, we have $\P[X_t^x \leq X_t^y] = 1$ for all $t\in[0,T]$.
        \item[(b)] For a given control $u \in L^{2}([0,T])$, let $x_u(\cdot;x)$ be the unique solution of
        \[
            x_u(t;x) = x + \int_0^t K(t-s)\left( b(s, x_u(s)) - \frac{1}{2}K(0)\sigma(s,x_u(s))\partial_x \sigma(s,x_u(s)) + \sigma(s, x_u(s))u(s) \right)\, \mathrm{d}s.
        \]
        Similarly, let $x_u(\cdot;y)$ be the unique solution where $x$ is replaced by $y$. Then
        \[
            x_u(t;x) \leq x_u(t;y), \qquad t \in [0,T].
        \]
    \end{enumerate}
\end{proposition}
\begin{proof}
    Let us consider the joint process $Z_t^{x,y} = (X_t^x, X_t^y)$ on $\R^2$. Let $\mu_{x,y} := \P \circ (Z^{x,y})^{-1}$ denote the law of $Z^{x,y}$ on $C([0,T]; \R^2)$ equipped with the supremum norm. Define the lower-diagonal of the path space as
    \[
        C_+([0,T]; \R^2) \coloneqq \left\{ (w_1,w_2) \in C([0,T]; \R^2) : w_1 \leq w_2 \right\},
    \]
    Since the evaluation map $(w_1,w_2) \longmapsto (w_1(t), w_2(t))$ is continuous with respect to the supremum norm, $C_+([0,T]; \R^2)$ is a closed subset of $C([0,T]; \R^2)$. By definition, $X^x \leq X^y$ a.s. is equivalent to $\mu_{x,y}( C_+([0,T]; \R^2) ) = 1$. Since the topological support of $\mu_{x,y}$ is the intersection of all closed sets in $C([0,T]; \R^2)$ that have full measure, we find
    \[
        \mu_{x,y}( C_+([0,T]; \R^2) ) = 1 \iff \mathrm{supp}(\mu_{x,y}) \subseteq C_+([0,T]; \R^2).
    \]
    An application of the support theorem \cite[Theorem 1.2]{K21} shows that $\mathrm{supp}(\mu_{x,y}) = \overline{\mathcal{S}}^{\| \cdot \|_\infty}$, where $\mathcal{S} = \left\{ (x_u(\cdot;x), x_u(\cdot;y)) \ : \ u \in L^2([0,T]) \right\}$. Hence, property (a) is equivalent to
    \[
        \mathcal{S} \subseteq \overline{\mathcal{S}}^{\| \cdot \|_\infty} \subseteq C_+([0,T]; \R^2),
    \]
    which readily implies (b). Conversely, condition (b) states that $\mathcal{S} \subseteq C_+([0,T]; \R^2)$, and taking the closure gives $\mathrm{supp}(\mu_{x,y}) = \overline{\mathcal{S}}^{\| \cdot \|_\infty} \subseteq C_+([0,T]; \R^2)$, and hence proves property (a).
\end{proof}

Below we focus on a special case that is of particular interest.

\begin{theorem}
    Let $b, \beta, \sigma, \sigma_0 \in \R$ with $\sigma \neq 0$, and $ K \in C^1([0,T])$. For $x \in \R$ let $X^x$ be the unique strong solution of
    \begin{align}\label{eq: 5}
        X^x_t = x + \int_0^t K(t-s)\left( b + \beta X_s^x\right) \, \mathrm{d}s + \int_0^t K(t-s)\left( \sigma_0 + \sigma X_s^x\right)\, \mathrm{d}B_s,\quad t\in[0,T].
    \end{align}
    Then the following are equivalent:
    \begin{enumerate}
        \item[(a)] (Comparison property) For any $x\leq y$, we have $\P[X_t^x \leq X_t^y] = 1$ for all $t\in[0,T]$.
        \item[(b)] For each control $u \in L^{2}([0,T])$, the unique solution of
        \begin{align}\label{eq: 4}
            x_u(t) = 1 + \int_0^t K(t-s)u(s) x_u(s)\, \mathrm{d}s, \quad t\in[0,T],
        \end{align}
        is nonnegative.
    \end{enumerate}
\end{theorem}
\begin{proof}
    Firstly, we find that $Y^{y-x}_t := X_t^y - X_t^x$ with $x \leq y$ satisfies
    \[
        Y_t^{y-x} = (y - x) + \beta\int_0^t K(t-s) Y_s^{y-x}\, \mathrm{d}s + \sigma \int_0^t K(t-s)Y_s^{y-x}\, \mathrm{d}B_s.
    \]
    Since this equation is linear, uniqueness yields $Y_t^{y-x} = (y-x)Y_t^1$. Hence, the comparison property (a) is satisfied if and only if $Y_t := Y_t^1$ is almost surely nonnegative.

    To further simplify the problem, we remove the drift $\beta$ via a change of measure. Define the shifted Brownian motion
    \[
        \mathrm{d}W_t = \mathrm{d}B_t + \frac{\beta}{\sigma}\mathrm{d}t,
    \]
    and introduce the equivalent probability measure $\mathbb{Q}$ via the Radon-Nikodym derivative
    \[
    Z_T := \frac{\mathrm{d}\mathbb{Q}}{\mathrm{d}\mathbb{P}}\bigg|_{\mathcal{F}_T} = \exp\left( -\frac{\beta}{\sigma}B_T - \frac{\beta^2}{2\sigma^2}T \right).
    \]
    By Girsanov's theorem, $W$ is a standard Brownian motion under $\mathbb{Q}$. Substituting $\mathrm{d}B_t = \mathrm{d}W_t - \frac{\beta}{\sigma}\mathrm{d}t$ into the equation for $Y_t$, the drift term cancels out which gives
    \[
        Y_t = 1 + \sigma \int_0^t K(t-s) Y_s \, \mathrm{d}W_s.
    \]
    Since $\P$ and $\mathbb{Q}$ are equivalent measures, $Y_t \geq 0$ $\P$-a.s. if and only if $Y_t \geq 0$ $\mathbb{Q}$-a.s.

    To characterise the latter, we use the Support Theorem \cite[Theorem 1.2]{K21}. Since $Y_t$ has an unbounded diffusion coefficient, such a result is not directly applicable, and an additional localisation argument needs to be used. Namely, let $\varphi_n\colon \R \longrightarrow \R$ is a smooth cutoff function such that $\varphi_n(y) = y$ for $|y| \leq n$, $|\varphi_n(y)| \leq n+1$ globally, with bounded derivatives, and suppose that $\varphi_n(y) \longrightarrow y$ holds locally uniformly as $n \to \infty$. Then
    \[
        Y_t^n = 1 + \sigma \int_0^t K(t-s) \varphi_n(Y_s^n)\, \mathrm{d}W_s,
    \]
    has for each $n \geq 1$ a unique strong solution. Because the coefficients are globally bounded, an application of the Support Theorem \cite[Theorem 1.2]{K21} gives
    \[
        \mathrm{supp}(\mathbb{Q} \circ (Y^n)^{-1}) = \overline{\mathcal{S}_n}^{\| \cdot \|_\infty},
    \]
    where $\mathcal{S}_n = \left\{ y_v^n  :  v \in L^2([0,T]) \right\}$ is the set of truncated skeleton paths solving
\[
    y_v^n(t) = 1 + \int_0^t K(t-s)\left(\sigma v(s) - \frac{\sigma^2}{2}K(0)\varphi_n'(y_v^n(s))\right)\varphi_n(y_v^n(s))\, \mathrm{d}s.
\]
    Because $\varphi_n'$ is bounded and the skeleton paths are continuous, the shift term $\frac{\sigma^2}{2}K(0)\varphi_n'(y_v^n(\cdot))$ is bounded and thus belongs to $L^2([0,T])$. Consequently, the mapping $v \longmapsto \widetilde{v}$ defined by
    \[
        \widetilde{v}(s) = \sigma v(s) - \frac{\sigma^2}{2}K(0)\varphi_n'(y_v^n(s))
    \]
    is a bijection on $L^2([0,T])$. Thus, we can absorb the correction term into the control variable and, without loss of generality, we may redefine $\mathcal{S}_n$ as the set of paths generated by the simplified equation
    \[
        y_v^n(t) = 1 +  \int_0^t K(t-s)v(s)\varphi_n(y_v^n(s))\, \mathrm{d}s.
    \]
    Define the open ball $U_n \coloneqq \left\{ w \in C([0,T]) : \ \|w\|_\infty < n \right\}$. By pathwise uniqueness, the untruncated process $Y$ and the truncated process $Y^n$ coincide up to the exit time of $U_n$. Therefore, their laws coincide on the open set $U_n$. Similarly, any solution $y_v$ of \eqref{eq: 4} that satisfies $y_v \in U_n$ also solves the truncated equation. Hence $\mathcal{S} \cap U_n = \mathcal{S}_n \cap U_n$, where $\mathcal{S}\coloneqq\{y_v : y_v \text{ solves \eqref{eq: 4} }, v\in L^2([0,T])\}$.

    Let $\mu \coloneqq \mathbb{Q} \circ Y^{-1}$ and $\mu_n \coloneqq \mathbb{Q} \circ (Y^n)^{-1}$.  Due to the Volterra-Gronwall-type inequality, linear Volterra equations do not exhibit finite-time blow-up, meaning that almost all paths of $Y$ are bounded, hence $\lim_{n \to \infty} \mu(U_n) = 1$. Let $V \subseteq C([0,T])$ be open. Then
    \[
        \mu(V) > 0 \ \iff \ \exists n \geq 1:  \mu_n(V \cap U_n) = \mu(V \cap U_n) > 0
    \]
    where the equality $\mu_n(V \cap U_n) = \mu(V \cap U_n)$ follows from $Y = Y^n$ on the event $\{t \leq \tau_n\}$, where $\tau_n = \inf\{ t > 0 \ : \ |Y_t| \geq n \}$. By definition of the support it is equivalent to the existence of $n \geq 1$ such that
    \begin{align}\label{eq: 2}
     \varnothing \neq (V \cap U_n) \cap \mathrm{supp}(\mu_n) = (V \cap U_n) \cap \overline{\mathcal{S}_n}^{\| \cdot \|_\infty}.
    \end{align}
    Since $V \cap U_n$ is an open set, it intersects the closure of $\mathcal{S}_n$ if and only if it intersects $\mathcal{S}_n$ itself. Thus \eqref{eq: 2} is equivalent to the existence of $n \geq 1$ such that
    \begin{align*}
        (V \cap U_n) \cap \mathcal{S}_n \neq \varnothing
        \iff V \cap (\mathcal{S} \cap U_n) \neq \varnothing.
    \end{align*}
    Since every continuous path in $\mathcal{S}$ is bounded, we have $\mathcal{S} \subset \bigcup_{n \in \N} U_n$. Therefore, the existence of such an $n$ is equivalent to
    \[
    V \cap \mathcal{S} \neq \varnothing \iff V \cap \overline{\mathcal{S}}^{\| \cdot \|_\infty} \neq \varnothing,
    \]
    where the second equivalence uses that $V$ is open. To summarise, by the definition of topological support, we conclude that $\mu(V) > 0$ if and only if $V \cap \overline{\mathcal{S}}^{\| \cdot \|_{\infty}} \neq \emptyset$ for all open sets $V$, and hence
    \[
        \mathrm{supp}(\mu) = \overline{ \mathcal{S}}^{\| \cdot \|_\infty}.
    \]
    Finally, to prove the desired equivalence, note that $Y_t\geq0$ $\mathbb{Q}$-a.s. is equivalent to $\mu(C([0,T];\R_+)=1$, and hence by definition of the support, also $\mathrm{supp}(\mu) \subseteq C([0,T]; \R_+)$. Hence, by the above, the comparison property (a) is equivalent to
    \[
    \mathcal{S}\subseteq\overline{\mathcal{S}}^{\|\cdot\|_\infty}=\operatorname{supp}(\mu)\subseteq C([0,T];\R_+)
    \]
    which readily implies (b). Vice versa, condition (b) states that $\mathcal{S}\subseteq C([0,T];\R_+)$. Taking the closure and using that $C([0,T];\R_+)$ is a closed subset of $C([0,T];\R)$, yields property (a).
\end{proof}

Firstly, these characterisations do not contradict our previous section. Indeed, if $K$ is additionally completely monotone, then condition (b) is automatically satisfied by \cite[Theorem C.2]{affine_volterra}, and hence \eqref{eq: 5} satisfies the comparison property.

\begin{remark}[Volterra stochastic exponential]
The Volterra stochastic exponential $\mathcal{E}_t = \mathcal{E}_t(K)$ is defined as the unique solution of
\[
        \mathcal{E}_t = 1 + \int_0^t K(t-s)\mathcal{E}_t\, \mathrm{d}B_s.
\]
The proof of the previous Theorem shows that $\mathcal{E}_t \geq 0$ a.s. if and only if each solution of $x_u(t) = 1 + \int_0^t K(t-s)u(s)x_u(s)\, \mathrm{d}s$ is nonnegative, where $u \in L_{\mathrm{loc}}^2(\R_+)$. In contrast to the Markov case, memory may break the nonnegativity of the stochastic exponential.
\end{remark}

To verify that for a given Volterra kernel, the comparison principle fails, it suffices to find one choice of control $u$ for which \eqref{eq: 4} does not have a nonnegative solution. The next remark provides a simple way to verify this.

\begin{remark}
    Suppose that $\int_0^{\infty} \mathrm{e}^{-\varepsilon t}K(t)\, \mathrm{d}t < \infty$ for each $\varepsilon > 0$. Choosing $u(t) =  -M$ with $M>0$ reduces \eqref{eq: 4} to the resolvent equation
    \[
        x_M(t) + M\int_0^t K(t-s)x_M(s)\, \mathrm{d}s = 1.
    \]
    Taking Laplace transforms, by Bernstein's theorem, we find that $x_M \geq 0$ almost everywhere if and only if the function
    \[
        (0,\infty) \ni z \longmapsto \widehat{x}_M(z) = \frac{1}{z(1+ M\widehat{K}(z))}
    \]
    is completely monotone. Moreover, let $R_M$ be the resolvent of the second kind associated to $M K$, solving $R_M + M K \ast R_M = MK$. Then, by variation of constants,
    \[
        x_M(t) = 1 - \int_0^t R_M(s)\, \mathrm{d}s
    \]
    and hence $x_M \geq 0$ if and only if $\int_0^{t} R_M(s)\, \mathrm{d}s \leq 1$ for all $t \geq 0$.
\end{remark}

\begin{example}
    Let $K(t) = 1 - \mathrm{e}^{-t}$. This kernel is non-negative, smooth, and bounded, but it is not completely monotone. Choosing $u(t) = -M$ with $M>0$, the deterministic skeleton equation becomes
    \[
        x_M(t) + M\int_0^t (1 - \mathrm{e}^{-(t-s)})x_M(s)\, \mathrm{d}s = 1.
    \]
    Differentiating this equation twice with respect to $t$ yields the second-order ODE
    \[
        x_M''(t) + x_M'(t) + M x_M(t) = 0, \qquad x_M(0) = 1, \quad x_M'(0) = 0.
    \]
    Take $M > 1/4$, then its characteristic roots given by $r = -\frac{1}{2} \pm \mathrm{i}\omega$ with $\omega = \sqrt{M - 1/4}$ are complex. The unique solution is given by
    \[
        x_M(t) = \mathrm{e}^{-t/2} \left( \cos(\omega t) + \frac{1}{2\omega} \sin(\omega t) \right).
    \]
    Because of the trigonometric terms, this path oscillates and takes negative values (e.g., $x_M(t) < 0$ at $t = \pi/\omega$). Hence the comparison property does not hold for \eqref{eq: 5}.
\end{example}

\begin{example}[fractional kernel]
    Consider $K(t) = t^\alpha$ with $\alpha > 1$ so that $K$ is continuously differentiable. Setting $u(t) = -M$ such that $M>0$, the deterministic controlled equation becomes
    \[
        x_M(t) + M\int_0^t (t-s)^\alpha x_M(s)\, \mathrm{d}s = 1,
    \]
    and has the unique solution
    \[
        x_M(t) = E_{\alpha+1, 1}\left( -M\Gamma(\alpha+1) t^{\alpha+1} \right),
    \]
    where $E_{\beta, 1}(z) := \sum_{k=0}^\infty \frac{z^k}{\Gamma(\beta k + 1)}$ is the Mittag-Leffler function. By \cite[Proposition 3.10]{mainardi_book} and the argument used in Example \ref{example_frac_OU}, the function $x \mapsto E_{\beta, 1}(-x)$ is non-negative for $x \geq 0$ if and only if $0 \leq \beta \leq 1$. In our case, $\beta = \alpha + 1 > 1$, and hence $x_M$ attains negative values. Thus \eqref{eq: 5} does not satisfy the comparison principle.
\end{example}

\begin{example}
    One might expect that if $K$ is strictly positive and decreasing, the memory effect decays smooth enough to preserve the comparison principle. However, the example $K(t) = 2\mathrm{e}^{-t} - \mathrm{e}^{-2t}$ shows that this is not the case. Clearly $K \geq 0$ and $K' \leq 0$, but $K$ is not completely monotone. Choosing a constant control $u(t) = -M$, the Laplace transform of the kernel is $\widehat{K}(z) = \frac{2}{z+1} - \frac{1}{z+2} = \frac{z+3}{(z+1)(z+2)}$. Substituting this into the resolvent equation yields the Laplace transform
    \[
        \widehat{x}_M(z) = \frac{1}{z\left(1 + M \frac{z+3}{z^2+3z+2}\right)} = \frac{z^2+3z+2}{z(z^2 + (3+M)z + (2+3M))}.
    \]
    This function is not completely monotone for $M = 2$ and hence $x_M$ is not nonnegative. Indeed, if it were completely monotone, its poles would be located on the negative real line, see \cite[Chapter 5, Theorem 2.6]{grippenberg}. However, the roots of the quadratic denominator $z^2 + 5z + 8 = 0$ are given by $z = -\frac{5}{2} \pm i\frac{\sqrt{7}}{2}$. Hence \eqref{eq: 5} does not satisfy the comparison principle.
\end{example}

\subsection{Sufficient conditions for the comparison principle}

Next we move towards stochastic Volterra equations with general initial conditions $g(t)$. Note that, without further assumptions on the drift, a comparison principle may fail. Indeed, let $X^i$ with $i = 1,2$ be the unique solution of
    \[
     X_t^i = x_i \frac{t^{\beta - 1}}{\Gamma(\beta)}  - \int_0^t \frac{(t-s)^{\alpha-1}}{\Gamma(\alpha)}X_s^i\,\d s + \int_0^t \frac{(t-s)^{\alpha-1}}{\Gamma(\alpha)} \,\d B_s
    \]
    where $\alpha > 1$, $\beta > 0$ and $x_1,x_2 \in \R$. Recall that by Example \ref{example_frac_OU}, the comparison principle holds if and only if $E_{\alpha,\beta}(-t^{\alpha}) \geq 0$ for all $t \geq 0$. If $0 < \beta < \alpha$ and $\alpha>1$, then $E_{\alpha,\beta}(-t^{\alpha})$ necessarily takes negative values, and hence the comparison principle fails.

Remark that in the above example, the drift $b(x) = -x$ is non-increasing. The following is our general result on the comparison property for regular kernels.

\begin{theorem}\label{comparison_principle_Dublin1}
    Suppose $K\in W^{1,2}_\loc(\R_+)$ such that $K(0)\geq0$ and $K'\geq0$\footnote{By abuse of terminology, we denote by $K$ its continuous representative and by $K'$ its weak derivative.}. Let $b_1,b_2 \colon  \R \longrightarrow \R$ be Lipschitz continuous, $\sigma \in \R$, and let $X^1,X^2$ be the unique solutions of
    \[
        X_t^i = g_i(t) + \int_0^t K(t-s)b_i(X_s^i)\, \mathrm{d}s + \int_0^t K(t-s)\sigma\, \mathrm{d}B_s.
    \]
    Moreover, suppose that the following conditions hold:
    \begin{enumerate}
        \item[(i)] the ordering condition $b_1(x)\leq b_2(x)$ holds for every $x\in\R$;
        \item[(ii)] at least one of the functions $b_1(\cdot),b_2(\cdot)$ is monotonically non-decreasing;
        \item[(iii)] $g_1,g_2\in C((0,\infty)$ satisfy $g_1\leq g_2$.
    \end{enumerate}
    Then $\P[X^1_t\leq X^2_t,\;t > 0] = 1$.
\end{theorem}
\begin{proof}
 Let us first note that the equations of interest have weak solutions $X^1,X^2$ with sample paths in $L^2_\loc(\R_+)$ due to \cite[Theorem 1.2]{weak_solution}. Moreover, it follows from \cite[Theorem 6.1 (ii)]{weak_solution} that $X^1-g_1,X^2-g_2$ have a modification with continuous sample paths. Hence, $X^1,X^2$ are continuous on $(0,\infty)$. Additionally,  these solutions satisfy for any $T>0, i=1,2$, $\int_0^T\E[|X_t^i|^2]\,\d t<\infty$ c.f. Proposition \ref{prop: existence}. The Yamada-Watanabe type of approach as given in \cite[Proposition B.3]{AbiJaber_ElEuch} yields pathwise uniqueness and hence the Yamada-Watanabe-Engelbert theorem (see Kurtz \cite[Theorem 1.5]{kurtz}) implies strong existence. Without loss of generality, we may assume that $b_1$ is non-decreasing. Since $K \in W_{\loc}^{1,2}(\R_+)$, the classical and stochastic version of Fubini's theorem \cite{veraar_fubini} and \cite[Theorem 8.2]{brezis2010functional}, imply that
   \begin{align*}
	X^1_t-X^2_t = \int_0^t\int_0^s K'(s-r)(b_1(X^1_r)-b_2(X^2_r))\,\d r\,\d s
			+K(0)\int_0^t (b_1(X^1_s)-b_2(X^2_s))\,\d s
			+(g_1-g_2)(t)
	\end{align*}
	which is a continuous semimartingale. Since $K(0)>0$ and $b_1\leq b_2$, and given that $b_1$ is both non-decreasing and Lipschitz continuous, we obtain
		\begin{align*}
			&K(0)\int_0^t (b_1(X^1_s)-b_2(X^2_s))\,\d s 	\\
            &\quad= K(0)\int_0^t (b_1(X^1_s)-b_1(X^2_s))\,\d s + K(0)\int_0^t (b_1(X^2_s)-b_2(X^2_s))\,\d s\\
            &\quad\leq K(0)\int_0^t \1_{\{X^1_s\geq X^2_s\}}(b_1(X^1_s)-b_1(X^2_s))\,\d s + K(0)\int_0^t \1_{\{X^1_s< X^2_s\}}(b_1(X^1_s)-b_1(X^2_s))\,\d s\\
            &\quad \lesssim  K(0)\int_0^t \1_{\{X^1_s\geq X^2_s\}}|X^1_s-X^2_s|\,\d s\\
            &\quad=  K(0)\int_0^t (X^1_s-X^2_s)_+\,\d s.
            \end{align*}
        In the same spirit and by using $K'\geq0$ and Fubini's theorem, we obtain
		\begin{align*}
			&\int_0^t\int_0^s K'(s-r)(b_1(X^1_r)-b_2(X^2_r))\,\d r\,\d s\\   &\quad= \int_0^t \int_0^s K'(s-r) (b_1(X^1_r)-b_1(X^2_r))\,\d r\,\d s
            +\int_0^t\int_0^s K'(s-r) (b_1(X^2_r)-b_2(X^2_r))\,\d r\,\d s
            \\ &\quad\leq \int_0^t\int_0^s K'(s-r) \1_{\{X^1_r\geq X^2_r\}}(b_1(X^1_r)-b_1(X^2_r))\,\d r\,\d s
			\\ &\quad\lesssim  \int_0^t\int_0^s K'(s-r)(X^1_r-X^2_r)_+\,\d r\,\d s
            \\ &\quad= \int_0^t\int_r^t K'(s-r)(X^1_r-X^2_r)_+\,\d s\,\d r
            \\ &\quad= \int_0^t (K(t-r)-K(0))(X^1_r-X^2_r)_+\,\d r.
		\end{align*}
    Consequently, by noting that $g_1-g_2\leq0$, we find that
    $\E[X^1_t-X^2_t] \lesssim \int_0^t K(t-s)\E[(X^1_s-X^2_s)_+]\,\d s$ and since the right-hand side is nonnegative, this implies
    \[\E[(X^1_t-X^2_t)_+] \lesssim \int_0^t K(t-s)\E[(X^1_s-X^2_s)_+]\,\d s.\]
    Finally, Gronwall's inequality for Volterra equations see e.g. \cite[Lemma A.1]{BBF23} implies $\E[(X^1_t-X^2_t)_+]=0$ and by the continuity of the sample paths we conclude $\P[X^1_t\leq X^2_t,\,t>0]=1$.

  \end{proof}

\begin{remark}
    The proof shows that the comparison principle also holds for the kernel $\widetilde{K}\coloneqq -K$ under the new assumption that $b_1(x)\geq b_2(x)$ and at least one of $b_1(\cdot),b_2(\cdot)$ is non-increasing. Furthermore, unlike in Section \ref{section: Non-negativity preserving case}, the presence of additive noise eliminates the need to impose a monotonicity assumption on $g_2-g_1$.
\end{remark}

The fractional kernel $K(t)=\frac{t^{\alpha-1}}{\Gamma(\alpha)}$ with $\alpha>3/2$ satisfies the assumptions stated in Theorem \ref{comparison_principle_Dublin1}. The case $\alpha>1$ can be obtained by approximation similar to Theorem \ref{comparison_cm_lg}. Since uniqueness is not guaranteed in such cases, we show the existence of a monotone coupling.

\begin{theorem}\label{comparison_regular_coupling}
    Suppose that $K \in C^1((0,\infty))$ with $K'\geq0$ and let $(g_i,K,b_i,\sigma)$, $i=1,2,$ satisfy Assumption \ref{assumption}. Suppose that assumptions (i) -- (iii) from Theorem \ref{comparison_principle_Dublin1} are satisfied. Then there exists a continuous weak solution $((\widehat{X}^1,\widehat{X}^2),$ $(\widehat{\Omega},\widehat{\F},\widehat{\mathbb{F}},\widehat{\P}),\widehat{B})$ of the two dimensional stochastic Volterra equation
    \begin{align*}\label{eq: joint coupling}
       \begin{pmatrix}
           \widehat{X}^1_t\\\widehat{X}^2_t
       \end{pmatrix}=
       \begin{pmatrix}
           g_1(t)\\g_2(t)
       \end{pmatrix}+\int_0^t K(t-s)
       \begin{pmatrix}
           b_1(\widehat{X}^1_s)\\b_2(\widehat{X}^2_s)
       \end{pmatrix}\d s + \sigma\int_0^t K(t-s)
       \begin{pmatrix}
           \widehat{X}^1_s\\ \widehat{X}^2_s
       \end{pmatrix}\d\widehat{B}_s
    \end{align*}
    such that $\widehat{\P}[\widehat{X}^1_t\leq \widehat{X}^2_t,\,t>0]=1$.
\end{theorem}
\begin{proof}
    We choose $(b_{1,n})_{n\in\N}$, $(b_{2,n})_{n\in\N}$ as in the proof of Theorem \ref{comparison_cm_lg} and can assume that $b_{1,n}$ or $b_{2,n}$ is non-decreasing. Moreover, define $K_n(t)=K(t+\frac{1}{n}).$ Clearly, $K_n\in W^{1,2}_\loc(\R_+)$ and $K_n\longrightarrow  K$ in $L^2_\loc(\R_+).$ Let $(X^n)_{n\in\N}$, $(Y^n)_{n\in\N}$ be the family of continuous processes constructed by $(g_1,K_n,b_{1,n},\sigma)$ and $(g_2,K_n,b_{2,n},\sigma)$ on some fixed filtered probability space $(\Omega,\F,\mathbb{F},\P)$. Then, by Theorem \ref{comparison_principle_Dublin1}, $\P[X^n_t\leq Y^n_t,\;t>0].$ Using Propsition \ref{prop_loc_hoelder}, we obtain for some $p\geq2$ suitably chosen
    \[
    \sup_{n\in\N}\E[|(X^n_t-g_1(t))-(X^n_s-g_2(s))|^p] + \sup_{n\in\N}\E[|(Y^n_t-g_1(t))-(Y^n_s-g_2(s))|^p] \leq c|t-s|^{1+\theta},
    \]
    where $\theta>0$ and $c$ is a constant that only depends on $p,C,T,K_1$. Define for every $n\in\N$ the continuous semimartingale
    \[
    Z^n_t = \int_0^t \begin{pmatrix}
    b_{1,n}(X^n_s)\\b_{2,n}(Y^n_s)
    \end{pmatrix}\d s + \sigma\int_0^t \begin{pmatrix}
    X^n_s\\ Y^n_s
    \end{pmatrix}\d B_s.
    \]
    Then we also obtain for some $p \geq 2$ suitably chosen
    \[
    \sup_{n\in\N}\E[|Z^n_t-Z^n_s|^p] \leq c|t-s|^{1+\theta},
    \]
    where $\theta>0$ and $c$ is a constant that only depends $C,T,K_1$ and $p$. Hence, we conclude by Kolmogorov's tightness criterion that $((X^n-g_1,Y^n-g_2,Z^n))_{n\in\N}$ is a tight sequence of continuous processes in $\R^4$ and so has a weakly convergent subsequence $((X^n-g_1,Y^n-g_2,Z^n))_{n\in\N}$ on the path space $C(\R_+;\R^4)$.

    Define the continuous mapping $\Phi\colon C(\R_+;\R^2)\times C(\R_+;\R^2)\longrightarrow  C_\delta(\R_+;\R^2)\times C(\R_+;\R^2)$ by
    \[
        \Phi(x,y,z_1,z_2) = (x+g_1,y+g_2,z_1,z_2).
    \]
    By the continuous mapping theorem, $(X^n,Y^n,Z^n)=\Phi(X^n-g_1,Y^n-g_2,Z)\Longrightarrow\widetilde{\P}$ on $C_\delta(\R_+;\R^2)\times C(\R_+;\R^2)$. As in the proof of step 2 in Theorem \ref{comparison_cm_lg}, we can show $\overline{\P}[x(t)\leq y(t),\;t>0]=1$ where $(x,y)$ denotes the projection to the first two coordinates of the coordinate process. Note that $C_\delta(\R_+;\R^2)\times C(\R_+;\R^2)$ is continuously embedded into $L^2_\loc(\R_+;\R^2)\times D(\R_+;\R^2)$ where we have equipped the Skorokhod space $D(\R_+;\R^2)$ with the usual Skorokhod topology, see e.g. \cite[Section 12]{billingsley}. Hence by passing to a subsequence $(X^n,Y^n,Z^n)\Longrightarrow  (\widehat{X},\widehat{Y},\widehat{Z})$ weakly in $L^2_\loc(\R_+;\R^2)\times D(\R_+;\R^2)$. Using \cite[Theorem 1.6]{weak_solution} yields that $(\widehat{X},\widehat{Y},\widehat{Z})$ is a weak solution for the data $((g_1,g_2),K,$ $(b_1,b_2),(\sigma,\sigma)))$ and satisfies $\widehat{\P}[\widehat{X}_t\leq \widehat{Y}_t,t>0]=1.$
\end{proof}

We expect that the results presented in section \ref{section: Regular Case} can be extended to time-dependent coefficients. However, establishing such an extension would require extending the results in \cite{weak_solution} to time-dependent coefficients, which is beyond the scope of this paper, and leave such an extension to the interested reader.

\begin{appendices}

    \section{Existence and uniqueness for Lipschitz coefficients}\label{appendix: Existence and Uniqueness}

    Below we prove a result on the existence and uniqueness of solutions for a stochastic Volterra equation with Lipschitz continuous coefficients and a general driving force $g \in L^q_\loc(\R_+)$, $q\in[2,\infty]$. The following extends \cite[Theorem 4.8]{BBF23} to time-dependent coefficients while, for the sake of simplicity, we stay in the finite-dimensional case.

    \begin{proposition}\label{prop: existence}
        Let $K \in L^2_\loc(\R_+)$ and let $b, \sigma\colon \R_+ \times \R \longrightarrow  \R$ be measurable such that
        \begin{equation}\label{eq:appendixA Lipschitz}
         |b(t,x) - b(t,y)| + |\sigma(t,x) - \sigma(t,y)| \leq C|x-y|,
        \end{equation}
        and
        \begin{equation}\label{eq:ex_uni_growth}
         |b(t,x)| + |\sigma(t,x)| \leq C (1+|x|)
        \end{equation}
        hold for a.a. $t \geq0$ and all $x,y \in \R$ with some constant $C > 0$. Then for each $g \in L^q_\loc(\R_+)$, $q\in[2,\infty]$, there exists a unique strong solution of the equation \eqref{eq:sve}. This solution satisfies $X \in L^p(\Omega, \F, \P; L_{\loc}^q(\R_+))$ for each $p \in [2,q]$ with $p < \infty$. To be more precise, $\norm{X}_{L^p(\Omega;L^q([0,T]))}\leq c$, where $c>0$ is a constant that only depends on $K,T,p,C$\footnote{The dependence on $C$ only comes through the growth estimate \eqref{eq:ex_uni_growth}.}.
    \end{proposition}
    \begin{proof}
        Below we follow the steps outlined in \cite[Theorem 4.8]{BBF23}. Fix $\lambda<0$ and define the functions $K_\lambda(t)=\e{\lambda t}K(t)$ and $g_\lambda(t) =  \e{\lambda  t}g(t)$. Let $X_t(g)$ be a solution of \eqref{eq:sve}. Then $Y^{\lambda}_t(g)=\e{\lambda t}X_t(g)$ satisfies
        \begin{equation}\label{eq:ex_uni_equiv}
            Y_t^{\lambda}(g) = g_\lambda(t) + \int_0^t K_\lambda(t-s)b_\lambda(s,Y_s^{\lambda}(g))\,\d s + \int_0^t K_\lambda(t-s)\sigma_\lambda(s,Y_s^{\lambda}(g))\,\d B_s,
        \end{equation}
        where $b_\lambda(s,x) = \e{\lambda s}b(s,\e{-\lambda s}x)$ and $\sigma_\lambda(s,x) = \e{\lambda s}\sigma(s,\e{-\lambda s}x)$. Conversely, let $Y_\lambda$ by a solution of \eqref{eq:ex_uni_equiv}, then $X^{\lambda}_t(g)=\e{-\lambda t}Y_\lambda(t;g)$ satisfies \eqref{eq:sve}. Therefore, it suffices to prove existence and uniqueness for \eqref{eq:ex_uni_equiv}. We solve \eqref{eq:ex_uni_equiv} by a fixed point argument for the case $p=q$.

        Fix $T>0$, then $g\vert_{[0,T]}\in L^q([0,T]).$ For a given $\lambda<0$, define
        \[
        \mathcal{T}_\lambda(X;g)(t)\coloneqq g_\lambda(t) + \int_0^t K_\lambda(t-s)b_\lambda(s,X_s)\,\d s + \int_0^t K_\lambda(t-s)\sigma_\lambda(s,X_s)\,\d B_s.
        \]
        Then \eqref{eq:ex_uni_equiv} is equivalent to $X^{\lambda} = \mathcal{T}_\lambda(X^{\lambda}(g))$. Below we show that $\mathcal{T_\lambda}(\cdot,g)$ is a contraction on $L^p(\Omega,\F,\P;L^p([0,T]))$ when $\lambda<0$ is small enough. For brevity, we let $\norm{\cdot}_p=\norm{\cdot}_{L^p(\Omega;L^p([0,T])}$. Take $X\in L^p(\Omega,\F,\P;L^p([0,T]))$, then
        \[
        \norm{\mathcal{T}_\lambda(X;g)}_p \leq \norm{g}_{L^p([0,T])} + \norm{K_\lambda\ast b_\lambda(\cdot,X)}_p + \norm{K_\lambda\ast \sigma_\lambda(\cdot,X)\,\d B}_p.
        \]
        For the drift we find, using Young's inequality
        \begin{align*}
            \norm{K_\lambda\ast b_\lambda(\cdot,X)}_p &= \norm*{ \norm{K_\lambda\ast b_\lambda(\cdot,X)}_{L^p([0,T])} }_{L^p(\Omega)} \\
            &\leq \norm{K_\lambda}_{L^1([0,T])} \,\norm*{ \norm{b_\lambda(\cdot,X_s)}_{L^p([0,T])} }_{L^p(\Omega)}\\
            &\leq C\norm{K_\lambda}_{L^1([0,T])}\norm{1+X}_p.
        \end{align*}
    For the stochastic convolution, we find using BDG and Young's inequality
    \begin{align*}
        \norm{K_\lambda\ast\sigma_\lambda(\cdot,X)\,\d B}_p &\leq c_p\left( \E\left[ \norm{K_\lambda^2\ast \sigma_\lambda(\cdot,X)^2}_{L^{p/2}([0,T])}^{p/2} \right] \right)^{1/p}\\
        &\leq c_p\norm{K^2_\lambda}_{L^1([0,T])}^{1/2}\left( \E\left[ \norm{\sigma_\lambda(\cdot,X)^2}_{L^{p/2}([0,T])}^{p/2} \right] \right)^{1/p}\\
        &\leq c_p\norm{K_\lambda}_{L^2([0,T])}\left( \E\left[ \norm{C(1+|X|)}_{L^p([0,T])}^p \right] \right)^{1/p}\\
        &\leq Cc_p\norm{K_\lambda}_{L^2([0,T])}\norm{1+X}_p.
    \end{align*}
    Hence, $\mathcal{T}_\lambda(\cdot;g)$ leaves the space $L^p(\Omega,\F,\P;L^p([0,T]))$ invariant. In the same way, if $X,Y\in L^p(\Omega,\F,\P;L^p([0,T]))$, we obtain
    \begin{align*}
    \norm{\mathcal{T}_\lambda(X;g)-\mathcal{T}_\lambda(Y;g)}_p &\leq \norm{K_\lambda\ast (b_\lambda(\cdot,X)-b_\lambda(\cdot,Y))}_p  + \norm{K_\lambda\ast (\sigma_\lambda(\cdot,X)-\sigma_\lambda(\cdot,Y))}_p\\
    &\leq C\max\{\norm{K_\lambda}_{L^1([0,T])},c_p\norm{K_\lambda}_{L^2([0,T])}\}\norm{X-Y}_p.
    \end{align*}

    Since by dominated convergence $\norm{K_\lambda}_{L^1([0,T])}, \ \norm{K_\lambda}_{L^2([0,T])} \longrightarrow  0$ as $\lambda\longrightarrow  -\infty$, we can choose $\lambda<0$ sufficiently small such that $\mathcal{T}_\lambda(\cdot;g)$ is a contraction and, hence, has a unique fixed point $Y_\lambda(\cdot;g)$ which is the unique solution of \eqref{eq:ex_uni_equiv}. The same fixpoint argument can be executed in the space $L^p(\Omega.\F,\P;L^{\infty}([0,T]))$, provided that $g \in L_\loc^{\infty}(\R_+)$. Using similar bounds as above, it is easy to derive, for $p\in[2,q]$ and a.a. $t\geq0$
        \[
        \E[|X_t|^p] \leq |g(t)|^p + A_0 + A_1\int_0^t K(t-s)^2\E[|X_s|^p]\,\d s,
        \]
        where $A_0, A_1$ are constants that only depends on $K,C,T,p$. Using the Volterra type Gronwall inequality (see e.g. \cite[Lemma A.1]{BBF23} yields for a.a. $t\geq0$
        \begin{equation}\label{eq:volterra_gronwall}
            \E[|X_t|^p]\leq |g(t)|^p + A_0 + \int_0^t R(t-s)(|g(s)|^p + A_0)\,\d s,
        \end{equation}
        where $R\in L^1([0,T])$ denotes the resolvent of the second kind of $A_1 K^2$. Since $|g|^p\in L^{q/p}([0,T])$, the right-hand side of \eqref{eq:volterra_gronwall} belongs to $L^{q/p}([0,T]).$ Thus, $\norm{X}_{L^p(\Omega;L^q([0,T]))}\leq c$, where $c$ is a constant that only depends on $K,C,T,p$.
  \end{proof}

    \begin{proposition}\label{prop_loc_hoelder}
        Suppose that $(g,K,b,\sigma)$ satisfies Assumption \ref{assumption}. Let $X$ be any solution of \eqref{eq:sve}. Then the following assertions hold:
        \begin{enumerate}
            \item Under Assumption \ref{assumption}.(i), $X-g$ has, for each $\theta<\gamma$, a continuous modification with locally $\theta$-H\"older continuous sample paths.
            \item Under Assumption \ref{assumption}.(ii), $X-g$ has, for each $\theta<\vartheta$ with
            \[
                 \vartheta \coloneqq \gamma + \frac{1}{2}\frac{\eta}{2+\eta} - \frac{2\xi}{q}\left(1 + \frac{1}{\eta}\right),
            \]
            a continuous modification with locally $\theta$-H\"older continuous sample paths.
        \end{enumerate}
    \end{proposition}
    \begin{proof}
    Firstly, let us consider the case (ii). Since $q \geq 2\xi\left( 1+ \frac{2}{\eta}\right)$, we find $p\coloneqq\frac{\eta q}{(2+\eta)\xi}\geq 2$. Then, for $0<s<t\leq T,$
        \begin{align*}
            &|(X_t-g(t))-(X_s-g(s))|^p\\ &\quad\lesssim \left| \int_s^t K(t-r)b(r,X_r)\,\d r \right|^p + \left| \int_0^s (K(t-r)-K(s-r))b(r,X_r)\,\d r \right|^p\\
            &\qquad+ \left| \int_s^t K(t-r)\sigma(r,X_r)\,\d B_r \right|^p + \left| \int_0^s (K(t-r)-K(s-r))\sigma(r,X_r)\,\d B_r \right|^p\\
            &\quad\eqqcolon I_1 + I_2 + I_3 + I_4
        \end{align*}
        An application of Jensen's and H\"older's inequality yields
        \begin{align}\label{eq:loc_hoelder1}
            \E[I_1] &\leq \left( \int_s^t |K(t-r)|\,\d r \right)^{p-1}\E\left[ \int_s^t |K(t-r)||b(r,X_r)|^p\,\d r \right]
            \\ \notag &\leq C^p\norm{K}_{L^1([0,t-s])}^{p-1}\E\left[\int_s^t |K(t-r)|(1+|X_r|)^{\xi p}\,\d r\right]
            \\ \notag &\leq C^p\norm{K}_{L^{2+\eta}([0,t-s])}^{p-1}(t-s)^{\frac{1+\eta}{2+\eta}(p-1)}\norm{K}_{L^{2+\eta}([0,t-s])}\E\left[\left(\int_s^t (1+|X_r|)^{\xi p\frac{2+\eta}{1+\eta}}\,\d r\right)^{\frac{1+\eta}{2+\eta}}\right]\\
            \notag &\leq C^pc(T,\eta)^p (t-s)^{\gamma p +  \frac{1+\eta}{2+\eta}(p-1)} \int_0^T \E\left[(1+|X_r|)^{q}\right]\d r
        \end{align}
        where the finiteness of the expectation follows from Proposition \ref{prop: existence}. Analogously, we obtain from It\^o's isometry, Jensen's  and H\"older's inequality
        \begin{align}
        \E[I_3] &= \E\left[ \left( \int_s^t K(t-r)^2\sigma(r,X_r)^2 \,\d r\right)^{\frac{p}{2}} \right]\nonumber\\
        &\leq C^p\left( \int_s^t K(t-r)^2\,\d r \right)^{\frac{p}{2}-1}\E\left[ \int_s^t K(t-r)^2(1+|X_r|)^{\xi p}\,\d r \right]\label{eq:loc_hoelder2}
        \\ &\leq C^p \norm{K}_{L^{2+\eta}([0,t-s])}^{p-2}(t-s)^{\frac{\eta}{2+\eta}\left(\frac{p}{2}-1\right)} \norm{K}_{L^{2+\eta}([0,t-s])}^2\E\left[\left( \int_s^t (1+|X_r|)^{\xi p\frac{2+\eta}{\eta}}\,\d r \right)^{\frac{2}{2+\eta}}\right]\nonumber\\
        &\leq C^p c(T,\eta)^p(t-s)^{\gamma p + \frac{\eta}{2+\eta}\left(\frac{p}{2}-1\right)} \int_0^T \E\left[(1+|X_r|)^{q}\right]\d r.\nonumber
        \end{align}
        The same estimate can be derived for $\E[I_2],\E[I_4].$ Thus, using $p \geq 2$, gives $(1+\eta)(p-1) \geq \eta \left(\frac{p}{2}-1\right)$ and hence we obtain
        \begin{align*}
            \left(\E[|(X_t-g(t)) - (X_s-g(s))|^p] \right)^{\frac{1}{p}}
            \leq c (t-s)^{\frac{1}{p}+\gamma + \frac{\eta}{2+\eta}\left(\frac{1}{2}- \frac{1}{p}\right)-\frac{1}{p}} = c(t-s)^{\frac{1}{p} + \vartheta},
        \end{align*}
        where $c>0$ is a constant that only depends on $p$,$\eta$, $C$, $T$ but not on $s$ or $t$, and the last equality follows from the particular form of $p$. Note that $\vartheta > 0$ if and only if
        \[
            q> 2\xi \frac{1 + \frac{1}{\eta}}{\gamma + \frac{1}{2}\frac{\eta}{2+\eta}}
        \]
        which is satisfied by Assumption \ref{assumption}.  The H\"older continuity now follows from the Kolmogorov-Chentsov theorem.

        For the case of Assumption \ref{assumption}.(i), we let $p \geq 2$ be arbitrary. By following similar arguments to \eqref{eq:loc_hoelder1} and \eqref{eq:loc_hoelder2} and using $\esssup_{r\in[0,T]}\E[|X_r|^p]<\infty$ due to Proposition \ref{prop: existence}, we find
        \[
            \E[|(X_t-g(t))-(X_s-g(s))|^p]\lesssim (t-s)^{p\gamma}.
        \]
        This proves the assertion.
        \end{proof}

        If $g\in L^\infty_\loc(\R_+)$, then the above result essentially coincides with \cite{affine_volterra}. For the case of Assumption \ref{assumption}.(ii), it gives a new way of proving the existence of continuous solutions.\\

        We conclude this section by establishing a stability result, demonstrating that strong solutions to SVEs depend continuously on the choice of the Volterra kernel.
    \begin{lemma}\label{lemma kernel approx}
        Suppose that $(g,K,b,\sigma)$ and $(g,\overline{K},b,\sigma)$ satisfy Assumption \ref{assumption}. Suppose further that the coefficients $b$ and $\sigma$ satisfy the Lipschitz conditions \eqref{eq:appendixA Lipschitz}. Let $X$ and $\overline{X}$ be the strong solutions to \eqref{eq:sve} for the given data. Then, for every $T>0$, there exists a constant $C>0$ such that
        \begin{equation*}
        \mathbb{E}[|X_t-\overline{X}_t|^2] \leq C \left( \int_0^T |K(s)-\overline{K}(s)|^{2+\eta} ds \right)^{\frac{2}{2+\eta}}, \quad t \in [0,T].
        \end{equation*}
    \end{lemma}
    \begin{proof}
        For fixed $t\in[0,T]$, we find by Jensen's inequality and It\^o's isometry
    \begin{align*}
        \E[|X_t-\overline{X}_t|^2] &\lesssim \int_0^t |K(t-s)-\overline{K}(t-s)|^2\E[b(s,X_s)^2]\,\d s\\
        &\quad+ \int_0^t \overline{K}(t-s)^2\E[|b_i(s,X_s)-b_i(s,\overline{X}_s)|^2]\,\d s \\
        &\quad+ \int_0^t |K(t-s)-\overline{K}(t-s)|^2 \E[\sigma(s,X_s)^2]\,\d s\\
        &\quad+ \int_0^t \overline{K}(t-s)^2\E[|\sigma(s,X_s)-\sigma(s,\overline{X}_s)|^2]\,\d s
        \\ &\lesssim \int_0^t |K(t-s)-\overline{K}(t-s)|^2\E[(1+|X_s|)^{2\xi}]\,\d s
        \\ &\qquad +  \int_0^t \overline{K}(t-s)^2 \E[|X_s-\overline{X}_s|^2]\,\d s
    \end{align*}
    where we have used the sublinear growth and Lipschitz condition. For the first term, we find by using H\"older's inequality
    \begin{align*}
    &\int_0^t |K(t-s)-\overline{K}(t-s)|^2\E[(1+|X_s|)^{2\xi}]\,\d s
    \\ &\quad\leq \left( \int_0^T |K(s)-\overline{K}(s)|^{2+\eta}\,\d s \right)^{\frac{2}{2 + \eta}}\E\left[\left(\int_0^T (1+|X_s|)^{2\xi\left(1 + \frac{2}{\eta}\right)}\,\d s \right)^\frac{\eta}{2 + \eta} \right]
    \\ &\quad \leq \left( \int_0^T |K(s)-\overline{K}(s)|^{2+\eta}\,\d s \right)^{\frac{2}{2 + \eta}} \int_0^T \E\left[ (1+|X_s|)^{2\xi\left(1 + \frac{2}{\eta}\right)} \right]\, \d s
    \end{align*}
    where the second term is finite due to Assumption \ref{assumption}
    \[
        q > 2\xi \frac{1 + \frac{1}{\eta}}{\gamma + \frac{1}{2}\frac{\eta}{2+\eta}}
        \geq 2\xi \frac{1+\frac{1}{\eta}}{\frac{1}{2} + \frac{1}{2}\frac{\eta}{2+\eta}} = 2\xi \left(1 + \frac{2}{\eta}\right)
    \]
    and Proposition \ref{prop: existence}. Applying the Volterra-type Gronwall inequality yields
    \begin{multline*}
        \E[|X_t - \overline{X}_t|^2]\\ \leq \left( \int_0^T |K(s)-\overline{K}(s)|^{2+\eta}\,\d s \right)^{\frac{2}{2 + \eta}} \left(\int_0^T \E\left[ (1+|X_s|)^{2\xi\left(1 + \frac{2}{\eta}\right)} \right] \d s\right) \left(1+\int_0^T R(s)\,\d s \right),
    \end{multline*}
    where $R$ denotes the resolvent of the second kind of $c\overline{K}^2$, and $c$ is some large enough constant. Noting that $R\in L^1_{\loc}(\R_+)$ proves the assertion.
    \end{proof}

\section{Splitting Method for Stochastic Volterra Equations}\label{appendix_splitting}
    Given any solution $X$ of \eqref{eq:sve} where $g\in L^2_\loc(\R_+)\cap C((0,\infty))$ and $b,\sigma\colon\R_+\times\R\longrightarrow\R$ are measurable such that there exists $C\in L^2_\loc(\R_+)$ with
    \begin{align}\label{eq: lipschitz L2}
        |b(t,x)-b(t,y)| + |\sigma(t,x)-\sigma(t,y)|\leq C(t)|x-y|,\quad |b(t,x)|^2 + |\sigma(t,x)|^2 \leq C^2(t)(1+|x|^2),
    \end{align}
    for a.a. $t\geq0$ and $x,y\in\R$. Let $T>0$, $N\in\N$ be fixed and $t_k= kT/N,$ $k\in\{0,1,\dots,N\},$ be an equidistant grid of $[0,T]$ with step size $T/N.$  Let us denote by $(\widehat{X}_t)_{t\in(0,T]}$ its approximation
    \begin{align}
    \widehat{X}_t &= g(t) + \sum_{k=1}^{N} \1_{[t_k,T]}(t)\left(K(t-t_k) \int_{t_{k-1}}^{t_k}[b(s,\xi^{k}_s)\,\d s + \sigma(s,\xi^{k}_s)\,\d B_s]\right)\nonumber
     \\
     &= g(t) + \int_0^{t_{\eta(t)}} K(t-t_{\eta(s)}) [b(s,\xi_s)\,\d s + \sigma(s,\xi_s)\,\d B_s]\label{eq:xhat_eta},
    \end{align}
    where $\eta(t)=k$ if $t\in[t_k,t_{k+1})$ and $\xi_t  \coloneqq \sum_{k=1}^N \1_{[t_{k-1},t_{k})}(t)\xi_t^{k}$ with
    \[
    \xi^{k}_t = \widehat{X}_{t_k-} + \int_{t_{k-1}}^t K(0_+)[b(s,\xi^{k}_s)\,\d s + \sigma(s,\xi^{k}_s)\,\d B_s],\quad t\in[t_{k-1},t_{k}).
    \]
    Denote by
    \[
        \omega_{f,T}(\delta)\coloneqq \sup_{s,t\in[0,T]:|s-t|\leq\delta}|f(s)-f(t)|
    \]
    the \emph{$\delta$-modulus of continuity} of a function $f\colon[0,T]\longrightarrow \R$ and by $\norm{f}_{[0,T]}\coloneqq\sup_{t\in[0,T]}|f(t)|$ its sup-norm. Below we extend the approximation \cite[Lemma 3.1, Proposition 3.1]{alfonsi2023} to the case of time-dependent coefficients and possibly singular $g$. As a first step, we prove a refined bound on the approximation where $g$ is regular.

    \begin{proposition}\label{prop_approx}
       Suppose that $K$ is nonnegative, nonincreasing and continuous on $\R_+$, and that $g \in C(\R_+)$  Then, for every $T>0$, it holds that
       \begin{equation}\label{eq:prob_approx_L2}
        \sup_{t\in[0,T]}\E[|\widehat{X}_t|^2 + |\xi_t|^2] \lesssim 1+\norm{g}_{[0,T]}^2
       \end{equation}
       and
        \begin{multline*}
           \sup_{t\in[0,T]}\left( \E[|\widehat{X}_t-\xi_t|^2] + \E[|\widehat{X}_t-X_t|^2] \right)\\
           \lesssim \omega^2_{g,T}\left(\frac{T}{N}\right) + \left( 1+ \|g\|_{[0,T]}^2\right)\left(\omega^2_{K,T}\left(\frac{T}{N}\right) + \sup_{1\leq k\leq N}\int_{t_{k-1}}^{t_k}C(s)^2\,\d s \right).
        \end{multline*}
    \end{proposition}
    \begin{proof}
        Let $k\in\{1,\dots,N\}$ and $t\in[t_{k-1},t_{k})$. By using \eqref{eq:aux_sde},
        \[
        |\xi_t|^2 \lesssim |\widehat{X}_{t_k-}|^2 + \left|\int_{t_{k-1}}^t b(s,\xi_s)\,\d s\right|^2 + \left|\int_{t_{k-1}}^t \sigma(s,\xi_s)\,\d B_s\right|^2.
        \]
        An application of Jensen's inequality, \Ito's isometry and the growth condition on $b,\sigma$ yields
        \begin{align}\label{eq: L2 approximation moment bound}
        \E[|\xi_t|^2] \lesssim \E[|\widehat{X}_{t_k-}|^2] + \int_{t_{k-1}}^t C(s)^2(1 + \E[|\xi_s|^2])\,\d s.
        \end{align}
        Thus by Gronwall's inequality, $\E[|\xi_t|^2] \lesssim \E[|\widehat{X}_{t_k-}|^2] + \int_{t_{k-1}}^{t_k} C(s)^2\,\d s$ for $t \in [t_{k-1},t_k)$. By letting $t\nearrow t_k$ in \eqref{eq:xhat_eta} and using $\lim_{t\nearrow t_k}\eta(t)=k-1$, we obtain
        \begin{align*}
            \E[|\widehat{X}_{t_k-}|^2] &\lesssim g(t_k)^2 + \int_0^{t_{k-1}} K(t-t_{\eta(s)})^2 C(s)^2(1+\E[|\xi_s|^2])\,\d s \\
            &\lesssim \norm{g}_{[0,T]}^2+\int_0^{t_{k-1}}C(s)^2\,\d s+\sum_{j=1}^{k-1}\int_{t_{j-1}}^{t_j}\left(\E[|\widehat{X}_{t_j-}|^2] + \int_{t_{j-1}}^{t_j}C(s)^2\,\d s\right)C(s)^2\,\d s\\
            &\lesssim 1+\norm{g}_{[0,T]}^2 + \sum_{j=1}^{k-1} \E[|\widehat{X}_{t_j-}|^2]\int_{t_{j-1}}^{t_j}C(s)^2\,\d s
        \end{align*}
        where we have used $\int_{t_{j-1}}^{t_j}C(s)^2\, \d s \leq \|C\|_{L^2([0,T]}^2$.
        By virtue of the discrete Gronwall inequality \cite{discrete_gronwall} and the standard inequality $1+x\leq \e{x}$, we obtain $\sup_{1\leq k\leq N}\E[|\widehat{X}_{t_k-}|^2] \lesssim 1+\norm{g}_{[0,T]}^2<\infty.$ Consequently, $\sup_{t\in[0,T]}\E[|\xi_t|^2]\lesssim1+\norm{g}_{[0,T]}^2$ and so by construction of the approximation we arrive at \eqref{eq:prob_approx_L2}.

        Let $t\in[t_{k-1},t_k)$. We have $|\widehat{X}_t-\xi_t|^2\lesssim |\widehat{X}_t-\widehat{X}_{t_k-}|^2+|\widehat{X}_{t_k-}-\xi_t|^2$. For the second term, using \eqref{eq: L2 approximation moment bound} we arrive at
        \begin{align*}
            \E[|\widehat{X}_{t_k-}-\xi_t|^2] \lesssim \int_{t_{k-1}}^t C(s)^2(1+\mathbb{E}[|\xi_s|^2])\,\d s \lesssim \left( 1+ \|g\|_{[0,T]}^2\right)\sup_{1\leq k\leq N}\int_{t_{k-1}}^{t_k}C(s)^2\,\d s.
        \end{align*}
       For the first difference, we have
        \[
        \widehat{X}_{t}-\widehat{X}_{t_{k-}} = g(t)-g(t_k) + \int_0^{t_{k-1}} (K(t-t_{\eta(s)})-K(t_k-t_{\eta(s)}))[b(s,\xi_s)\,\d s + \sigma(s,\xi_s)\,\d B_s]
        \]
        and consequently
        \begin{align*}
            &\E[|\widehat{X}_t-\widehat{X}_{t_k-}|^2]
            \\ &\quad\lesssim |g(t)-g(t_k)|^2 + \int_0^{t_{k-1}}(K(t-t_{\eta(s)})-K(t_k-t_{\eta(s)}))^2C(s)^2(1+\E[|\xi_t|^2])\,\d s
            \\ &\quad\lesssim \omega^2_{g,T}\left(\frac{T}{N}\right) + \left( 1+ \|g\|_{[0,T]}^2\right)\omega^2_{K,T}\left(\frac{T}{N}\right).
        \end{align*}
        Finally, it remains to bound the difference
        \begin{align*}
            \widehat{X}_t -X_t &=  \int_0^t K(t-s)[b(s,\xi_s)-b(s,X_s)\,\d s + \sigma(s,\xi_s)-\sigma(s,X_s)\,\d B_s]\\
            &\quad+\int_0^{t_{\eta(t)}} (K(t-t_{\eta(s)})-K(t-s))[b(s,\xi_s)\,\d s + \sigma(s,\xi_s)\,\d B_s]\\
            &\quad-\int_{t_{\eta(t)}}^t K(t-s)[b(s,\xi_s)\,\d s + \sigma(s,\xi_s)\,\d B_s].
        \end{align*}
        Using the linear growth and Lipschitz condition \eqref{eq: lipschitz L2}, we obtain
        \begin{align*}
            \E[|\widehat{X}_t-X_t|^2] &\lesssim \int_0^t K(t-s)^2 C(s)^2\E[|\xi_s-X_s|^2]\,\d s\\
            &\quad+ \int_0^{t_{\eta(t)}} (K(t-t_{\eta(s)})-K(t-s))^2 C(s)^2(1+\E[|\xi_s|^2])\,\d s\\
            &\quad+ \int_{t_{\eta(t)}}^t K(t-s)^2C(s)^2(1+\E[|\xi_s|^2])\,\d s.
        \end{align*}
        For the first term we use $|\xi_s - X_s|^2\lesssim |\xi_s-\widehat{X}_s|^2 + |\widehat{X}_s-X_s|^2$ to bound
        \begin{align*}
           &\  \int_0^t K(t-s)^2 C(s)^2\E[|\xi_s-X_s|^2]\,\d s
            \\ &\lesssim \int_0^t K(t-s)^2 C(s)^2\E[|\xi_s-\widehat{X}_s|^2]\,\d s
             + \int_0^t K(t-s)^2 C(s)^2\E[|\widehat{X}_s-X_s|^2]\,\d s
            \\ &\lesssim \omega^2_{g,T}\left(\frac{T}{N}\right) + \left( 1+ \|g\|_{[0,T]}^2\right)\left(\omega^2_{K,T}\left(\frac{T}{N}\right) + \sup_{1\leq k\leq N}\int_{t_{k-1}}^{t_k}C(s)^2\,\d s\right) + \int_0^t C(s)^2 \E[|\widehat{X}_s-X_s|^2]\,\d s.
        \end{align*}
        For the second term, we use \eqref{eq: L2 approximation moment bound} to find that
        \begin{align*}
            \int_0^{t_{\eta(t)}} (K(t-t_{\eta(s)})-K(t-s))^2 C(s)^2(1+\E[|\xi_s|^2])\,\d s
            &\lesssim \left( 1+ \|g\|_{[0,T]}^2\right)\omega^2_{K,T}\left(\frac{T}{N}\right).
        \end{align*}
        Finally, the last term can be bounded by
        \begin{align*}
            \int_{t_{\eta(t)}}^t K(t-s)^2C(s)^2(1+\E[|\xi_s|^2])\,\d s
            &\lesssim \left( 1+ \|g\|_{[0,T]}^2\right)\sup_{1\leq k\leq N}\int_{t_{k-1}}^{t_k}C(s)^2\,\d s.
        \end{align*}
        Hence, collecting all estimates, Gronwall's inequality yields the assertion.
    \end{proof}

    Next, we prove the convergence of the approximation also in the case where $g$ may be singular in $t = 0$, provided that the singularity is not too rough.

    \begin{proposition}\label{prop_approx_sing_g}
        Let $g\in L^2_\loc(\R_+)\cap C((0,\infty))$ be such that there exists $\delta>0$ with
        \[
            \int_0^T \overline{g}(s)^{2+\delta}C(s)^2\,\d s <\infty
        \]
        where $\overline{g}(s)\coloneqq \sup_{s\leq t\leq T}|g(t)|$. Then, for every $t\in(0,T]$,
        \[
        \lim_{N\longrightarrow \infty}\E[|\widehat{X}_t-X_t|^2]= 0.
        \]
    \end{proposition}
    \begin{proof}
    Fix $\varepsilon>0$ and define $g_\varepsilon\colon\R_+\longrightarrow \R$ by $g_\varepsilon(t)=g(t+\varepsilon).$ Moreover, let $X^\varepsilon$ be the strong solution of
    \[
    X^\varepsilon_t = g_\varepsilon(t) + \int_0^t K(t-s)[b(s,X^\varepsilon_s)\,\d s + \sigma(s,X^\varepsilon_s)\,\d B_s].
    \]
    Denote by $\widehat{X}^\varepsilon$ its approximation and by $\xi^\varepsilon$ the process that arises in its construction. Then $|\widehat{X}_t-X_t|\leq |\widehat{X}_t-\widehat{X}^\varepsilon_t| + |\widehat{X}^\varepsilon_t-X^\varepsilon_t| + |X^\varepsilon_t-X_t|$. Proposition \ref{prop_approx} gives
    \[
    \E[|\widehat{X}^\varepsilon_t-X^\varepsilon_t|^2]
    \lesssim \omega^2_{g_\varepsilon,T}\left(\frac{T}{N}\right) + \left(1 + \|g_{\varepsilon}\|_{[0,T]}^2\right)\left( \omega^2_{K,T}\left(\frac{T}{N}\right) + \sup_{1\leq k\leq N}\int_{t_{k-1}}^{t_k} C(s)^2\,\d s\right),
    \]
    and it is a routine to show $\E[|X^\varepsilon_t-X_t|^2]\lesssim |g_\varepsilon(t)-g(t)|^2.$ Furthermore, we claim that
    \begin{equation}\label{eq:hatX-hatXeps}
        \E[|\widehat{X}_t-\widehat{X}^\varepsilon_t|^2]\lesssim |g(t)-g_\varepsilon(t)|^2+\sum_{j=1}^{\eta(t)}|g(t_j)-g_\varepsilon(t_j)|^2\int_{t_{j-1}}^{t_j}C(s)^2\,\d s, \qquad t \in (0,T].
    \end{equation}
    Indeed, the claim is certainly true on $(0,t_1)$ due to $\widehat{X}_t^{\varepsilon} = g_{\varepsilon}(t)$ and $\widehat{X}_t = g(t)$. Suppose \eqref{eq:hatX-hatXeps} were true on $(0,t_k)$ and let $t\in[t_k,t_{k+1})$ then using \eqref{eq: lipschitz L2} we arrive at
    \begin{align*}
        &\E[|\widehat{X}_t-\widehat{X}^\varepsilon_t|^2]\\&\quad\lesssim |g(t)-g_\varepsilon(t)|^2 + \sum_{j=1}^k \int_{t_{j-1}}^{t_j}C(s)^2\E[|\xi_s-\xi^\varepsilon_s|^2]\,\d s \\
        &\quad\lesssim |g(t)-g_\varepsilon(t)|^2 + \sum_{j=1}^{k-1}|g(t_j)-g_\varepsilon(t_j)|^2\int_{t_{j-1}}^{t_j} C(s)^2\,\d s + \int_{t_{k-1}}^{t_k} C(s)^2\E[|\xi_s-\xi^\varepsilon_s|^2]\,\d s.
    \end{align*}
    Finally, an application of Gronwall's inequality to estimate $\E[|\xi_s-\xi^\varepsilon_s|^2]\lesssim |g(t_k)-g_\varepsilon(t_k)|^2$ yields \eqref{eq:hatX-hatXeps}.
    Consequently, the following less sharp version of \eqref{eq:hatX-hatXeps} holds true
    \begin{align*}
    \E[|\widehat{X}_t- \widehat{X}_t^\varepsilon|^2]
    &\lesssim |g(t)-g_\varepsilon(t)|^2 + \sum_{j=1}^N|g(t_j)-g_\varepsilon(t_j)|^2\int_{t_{j-1}}^{t_j}C(s)^2\,\d s
    \\ &= |g(t)-g_\varepsilon(t)|^2 + \int_0^T G^N_\varepsilon(s)C^2(s)\,\d s,
    \end{align*}
    where $G^N_\varepsilon(s)\coloneqq\sum_{j=1}^N \1_{[t_{j-1},t_j)}(s)|g(t_j)-g_\varepsilon(t_j)|^2$. By continuity of $g$ on $(0,T]$, one can show that $G^N_\varepsilon(t)\overset{N\longrightarrow \infty}{\longrightarrow}|g(t)-g_\varepsilon(t)|^2$ for all $t\in(0,T]$ and fixed $\varepsilon > 0$. Moreover, it holds that
    \begin{align*}
        \int_0^T |G^N_\varepsilon(s)|^{1+\delta/2}C(s)^2\,\d s
        &= \sum_{j=1}^N \int_{t_{j-1}}^{t_j} |g(t_j)-g_\varepsilon(t_j)|^{2+\delta} C(s)^2\,\d s\\
        &\lesssim \sum_{j=1}^N \int_{t_{j-1}}^{t_j} |g(t_j)|^{2+\delta}C(s)^2\,\d s + \norm{g_\varepsilon}_{[0,T]}^{2+\delta}\int_0^T C(s)^2\,\d s\\
        &\leq\int_0^T \overline{g}(s)^{2+\delta}C(s)^2\,\d s + \norm{g_\varepsilon}_{[0,T]}^{2+\delta}\int_0^T C(s)^2\,\d s.
    \end{align*}
    Thus, $(G^N_\varepsilon)_{N\in\N}$ is uniformly integrable. Moreover,
    \[
    \lim_{N \longrightarrow  \infty}\sup_{1\leq k\leq N}\int_{t_{k-1}}^{t_k} C(s)^2\,\d s = 0.
    \]
    Indeed, $t\longmapsto\int_0^t C(s)^2\,\d s$  is an absolutely continuous function. Thus, for every $\widetilde{\varepsilon}>0$ there exists $\widetilde{\delta}>0$ such that for any $a,b>0$ with $b-a<\widetilde{\delta}$ holds $\int_a^b C(s)^2\,\d s<\widetilde{\varepsilon}$. Choosing $N$ large enough yields $t_k-t_{k-1}=\frac{1}{N}<\widetilde{\delta}$ and so $\int_{t_{k-1}}^{t_k}C(s)^2\,\d s<\widetilde{\varepsilon}$, $k\in\{1,\dots,N\}$. Collecting all estimates gives
    \begin{align*}
       \limsup_{N\longrightarrow \infty}\E[|\widehat{X}_t-X_t|^2]\lesssim |g(t)-g_\varepsilon(t)|^2  + \int_0^T |g(s)-g_\varepsilon(s)|^2 C(s)^2\,\d s.
    \end{align*}
    The right-hand side tends to zero as $\varepsilon\downarrow0$ due to dominated convergence, which completes the proof.
    \end{proof}

Finally, we prove an auxiliary result for kernels that preserve positivity.

    \begin{lemma}\label{prop_non-neg_preserv}
		Suppose $K\colon(0,\infty)\longrightarrow \R_+$ is a non-negativity preserving kernel with $K(0_+)>0$ and let $f\colon (0,\infty) \longrightarrow \R_+$ be non-decreasing. Let $N\in\N$, $0 < t_1<\dots < t_N$ and $x_1,\dots,x_N\in\R$ be such that, for any $k\in\{1,\dots,N\},$
		\begin{equation}\label{eq:assumption_pos_preserv}
			f(t_k) + \sum_{\ell=1}^k x_{\ell} K(t_k-t_{\ell})\geq0.
		\end{equation}
		Then, it follows
		\begin{equation*}
			f(t) + \sum_{k=1}^N \1_{\{ t_k\leq t\}} x_k K(t-t_k)\geq0,\quad t > 0.
		\end{equation*}
	\end{lemma}
	\begin{proof}
		We define recursively $\widetilde{x}_1\coloneqq\frac{-f(t_1)}{K(0_+)}$ and, for $k\in\{1,\dots,N\},$
		\[
		\widetilde{x}_k\coloneqq \frac{-1}{K(0_+)} \left( f(t_k) + \sum_{\ell=1}^{k-1}\widetilde{x}_{\ell}K(t_k-t_{\ell})\right).
		\]
		By construction we have for any $k\in\{1,\dots,N\}$
		\begin{equation} \label{eq:construction}
			f(t_k) + \sum_{\ell=1}^k \widetilde{x}_{\ell} K(t_k-t_{\ell}) =0.
		\end{equation}
		Moreover, define $\delta_k \coloneqq x_k - \widetilde{x}_k,\, k\in\{ 1,\dots,N \}.$
		Then, we have by \eqref{eq:construction} and \eqref{eq:assumption_pos_preserv}
		\[
		\sum_{\ell=1}^k \delta_{\ell}K(t_k-t_{\ell}) = \sum_{\ell=1}^k x_{\ell} K(t_k-t_{\ell})+f(t_k)\geq0,\quad k\in\{ 1,\dots,N \},
		\]
		and since $K$ preserves non-negativity, we obtain $	\sum_{k=1}^N \1_{\{ t_k\leq t\}} \delta_k K(t-t_k)\geq0$ for $t > 0$.
		We have
		\begin{align*}
			f(t) + \sum_{k=1}^N \1_{\{ t_k\leq t\}} x_k K(t-t_k) = f(t) + \sum_{k=1}^N \1_{\{ t_k\leq t\}} \widetilde{x}_k K(t-t_k) + \sum_{k=1}^N \1_{\{ t_k\leq t\}} \delta_k K(t-t_k)
		\end{align*}
		and hence, it suffices to show that
		\begin{equation}\label{eq:remaining_inequality}
			f(t) + \sum_{k=1}^N \1_{\{ t_k\leq t\}} \widetilde{x}_k K(t-t_k)\geq0,\quad t > 0.
		\end{equation}
        By induction over $k$ we show that \eqref{eq:remaining_inequality} holds on $[t_k,t_{k+1})$ and $\widetilde{x}_1,\dots,\widetilde{x}_{k+1}\leq0$. For $t\in (0,t_1)$, \eqref{eq:remaining_inequality} is evident and by assumption $\widetilde{x}_1 = \frac{-f(t_1)}{K(0_+)}\leq0$. Suppose \eqref{eq:remaining_inequality} holds on $[t_{k-1},t_k)$ and $\widetilde{x}_1,\dots,\widetilde{x}_{k}\leq0.$ Then, we find for $t\in[t_k,t_{k+1})$
        \begin{align*}
            f(t) + \sum_{\ell=1}^k \widetilde{x}_\ell K(t-t_\ell) &\geq f(t) + \sum_{\ell=1}^{k-1} \widetilde{x}_\ell K(t-t_\ell) - \left(f(t_k)+\sum_{\ell=1}^{k-1}\widetilde{x}_\ell K(t_k-t_\ell)\right)\\
            &= f(t)-f(t_k) +\sum_{\ell=1}^{k-1}\widetilde{x}_\ell \left(K(t-t_\ell)-K(t_k-t_\ell)\right)\\
            &\geq f(t)-f(t_k)\\&\geq0,
            \end{align*}
        where we used that $K(t-t_\ell)-K(t_k-t_\ell)\leq0$ since $K\geq0$ is non-increasing. Replacing $t$ by $t_{k+1}$ in the above calculation yields $\widetilde{x}_{k+1}\leq0.$
\end{proof}

\end{appendices}

	\bibliographystyle{plainurl}
	\bibliography{literature}

@article {MR4195178,
    AUTHOR = {Friesen, Martin and Jin, Peng and R\"{u}diger, Barbara},
     TITLE = {On the boundary behavior of multi-type continuous-state
              branching processes with immigration},
   JOURNAL = {Electron. Commun. Probab.},
  FJOURNAL = {Electronic Communications in Probability},
    VOLUME = {25},
      YEAR = {2020},
     PAGES = {Paper No. 84, 14},
      ISSN = {1083-589X},
   MRCLASS = {60G17 (60J25 60J80)},
  MRNUMBER = {4195178},
}

@article{berger_mizel,
 author = {Marc A. Berger and Victor J. Mizel},
 journal = {Journal of Integral Equations},
 number = {3},
 pages = {187--245},
 publisher = {Rocky Mountain Mathematics Consortium},
 title = {Volterra Equations with Itô Integrals—I},
 urldate = {2024-03-01},
 volume = {2},
 year = {1980}
}

@Book{ikeda_watanabe,
  author    = {Ikeda, Nobuyuki and Watanabe, Shinzo},
  publisher = {North-Holland Publishing Company Amsterdam - Oxford - New York},
  title     = {Stochastic Differential Equations and Diffusion Processes},
  year      = {1988},
  edition   = {2},
}

@article{AS24,
 AUTHOR = {Alfonsi, Aur\'elien and Szulda, Guillaume},
     TITLE = {On non-negative solutions of stochastic {V}olterra equations
              with jumps and non-{L}ipschitz coefficients},
   JOURNAL = {Bernoulli},
  FJOURNAL = {Bernoulli. Official Journal of the Bernoulli Society for
              Mathematical Statistics and Probability},
    VOLUME = {31},
      YEAR = {2025},
    NUMBER = {4},
     PAGES = {2890--2915},
      ISSN = {1350-7265,1573-9759},
   MRCLASS = {60H20 (45R05 91G80)},
  MRNUMBER = {4931355},
}

@article{BBF23,
    AUTHOR = {Bianchi, Luigi Amedeo and Bonaccorsi, Stefano and Friesen,
              Martin},
     TITLE = {Limits of stochastic {V}olterra equations driven by {G}aussian
              noise},
   JOURNAL = {Stoch. Partial Differ. Equ. Anal. Comput.},
  FJOURNAL = {Stochastics and Partial Differential Equations. Analysis and
              Computations},
    VOLUME = {13},
      YEAR = {2025},
    NUMBER = {2},
     PAGES = {585--630},
      ISSN = {2194-0401,2194-041X},
   MRCLASS = {60H15 (45D05 60B10 60F05 60G10 60G22 60H20)},
  MRNUMBER = {4908976},
}

@article{AbiJaber_ElEuch,
author = {Abi Jaber, Eduardo and El Euch, Omar},
title = {Multifactor Approximation of Rough Volatility Models},
journal = {SIAM Journal on Financial Mathematics},
volume = {10},
number = {2},
pages = {309-349},
year = {2019},
}

@article {alfonsi2022,
    AUTHOR = {Alfonsi, Aur\'{e}lien and Kebaier, Ahmed},
     TITLE = {Approximation of stochastic {V}olterra equations with kernels
              of completely monotone type},
   JOURNAL = {Math. Comp.},
  FJOURNAL = {Mathematics of Computation},
    VOLUME = {93},
      YEAR = {2024},
    NUMBER = {346},
     PAGES = {643--677},
      ISSN = {0025-5718,1088-6842},
   MRCLASS = {65R20 (45D05 60G22 60H35 91G60)},
  MRNUMBER = {4678580},
}

@article{alfonsi2023,
      title = {Nonnegativity preserving convolution kernels. Application to Stochastic Volterra Equations in closed convex domains and their approximation},
        journal = {Stochastic Processes and their Applications},
        volume = {181},
        pages = {104535},
        year = {2025},
        issn = {0304-4149},
        author = {Aurélien Alfonsi},
}

@book{SchillingSongVondracek+2012,
title = {Bernstein Functions - Theory and Applications},
author = {René L. Schilling and Renming Song and Zoran Vondracek},
publisher = {De Gruyter},
address = {Berlin, Boston},
isbn = {9783110269338},
year = {2012},
}

@article{affine_volterra,
author = {Eduardo {Abi Jaber} and Martin Larsson and Sergio Pulidogr},
title = {{Affine Volterra processes}},
volume = {29},
journal = {The Annals of Applied Probability},
number = {5},
publisher = {Institute of Mathematical Statistics},
pages = {3155 -- 3200},
year = {2019},
}

@book{revuz2004continuous,
  title={Continuous Martingales and Brownian Motion},
  author={Revuz, D. and Yor, M.},
  isbn={9783540643258},
  lccn={98053189},
  series={Grundlehren der mathematischen Wissenschaften},
  year={1999},
  publisher={Springer Berlin Heidelberg}
}

@Book{grippenberg,
  author    = {Grippenberg, G. and Staffans, S.-O. and Staffans, S.-O.},
  publisher = {Cambridge University Press},
  title     = {Volterra Integral and Functional Equations},
  year      = {1990},
  edition   = {34},
}

@book{billingsley,
  address = {New York},
  author = {Billingsley, Patrick},
  description = {q-paper},
  edition = {Second},
  isbn = {0-471-19745-9},
  mrclass = {60B10 (28A33 60F17)},
  mrnumber = {MR1700749 (2000e:60008)},
  note = {A Wiley-Interscience Publication},
  pages = {x+277},
  publisher = {John Wiley \& Sons Inc.},
  series = {Wiley Series in Probability and Statistics: Probability and
              Statistics},
  title = {Convergence of probability measures},
  year = 1999
}

@article{discrete_gronwall,
  title={Short proof of a discrete Gronwall inequality},
  author={Dean S. Clark},
  journal={Discret. Appl. Math.},
  year={1987},
  volume={16},
  pages={279-281},
}

@Article{protter_volterra,
  author    = {Protter,Philip},
  journal   = {The Annals of Probability},
  title     = {Volterra Equations Driven by Semimartingales},
  year      = {1985},
  issn      = {00911798},
  number    = {2},
  pages     = {519--530},
  volume    = {13},
  publisher = {Institute of Mathematical Statistics},
}

@article{weak_solution,
author = {Eduardo {Abi Jaber} and Christa Cuchiero and Martin Larsson and Sergio Pulido},
title = {{A weak solution theory for stochastic Volterra equations of convolution type}},
volume = {31},
journal = {The Annals of Applied Probability},
number = {6},
publisher = {Institute of Mathematical Statistics},
pages = {2924 -- 2952},
year = {2021},
}

@Article{kurtz,
  author  = {Thomas G. Kurtz},
  journal = {Electron. Commun. Probab. \textbf{19} (2014), no. 58, 1–16},
  title   = {Weak and strong solutions of general stochastic models},
  year    = {2014},
}

@book{brezis2010functional,
  title={Functional Analysis, Sobolev Spaces and Partial Differential Equations},
  author={Brezis, H.},
  isbn={9780387709130},
  lccn={2010938382},
  series={Universitext},
  year={2010},
  publisher={Springer New York}
}

@Article{veraar_fubini,
  author  = {Mark Veraar},
  journal = {Stochastics},
  title   = {The stochastic Fubini theorem revisited},
  year    = {2012},
}

@article{markovian_structure,
title = {Markovian structure of the Volterra Heston model},
journal = {Statistics and Probability Letters},
volume = {149},
pages = {63-72},
year = {2019},
author = {Eduardo {Abi Jaber} and Omar {El Euch}},
}

@article{proemel_dec22,
	AUTHOR = {Pr\"omel, David J. and Scheffels, David},
     TITLE = {Pathwise uniqueness for singular stochastic {V}olterra
              equations with {H}\"older coefficients},
   JOURNAL = {Stoch. Partial Differ. Equ. Anal. Comput.},
  FJOURNAL = {Stochastics and Partial Differential Equations. Analysis and
              Computations},
    VOLUME = {13},
      YEAR = {2025},
    NUMBER = {1},
     PAGES = {308--366},
      ISSN = {2194-0401,2194-041X},
   MRCLASS = {60H20 (45D05 60H15)},
  MRNUMBER = {4872110},
}

@Misc{mytnik,
  author    = {Mytnik, Leonid and Salisbury, Thomas S.},
  title     = {Uniqueness for Volterra-type stochastic integral equations},
  year      = {2015},
  copyright = {arXiv.org perpetual, non-exclusive license},
  eprint     = {1502.05513},
  publisher = {arXiv},
}

@Article{el_euch1,
  author   = {Omar El Euch and Masaaki Fukasawa and Mathieu Rosenbaum},
  journal  = {Finance and Stochastics},
  title    = {{The microstructural foundations of leverage effect and rough volatility}},
  year     = {2018},
  month    = {April},
  number   = {2},
  pages    = {241-280},
  volume   = {22},
}

@Article{el_euch2,
  author    = {Omar El Euch and Mathieu Rosenbaum},
  journal   = {The Annals of Applied Probability},
  title     = {PERFECT HEDGING IN ROUGH HESTON MODELS},
  year      = {2018},
  issn      = {10505164, 21688737},
  number    = {6},
  pages     = {3813--3856},
  volume    = {28},
  publisher = {Institute of Mathematical Statistics},
  }

@Article{pricing_under_rough_volatility,
  author    = {Christian Bayer and Peter Friz and Jim Gatheral},
  journal   = {Quantitative Finance},
  title     = {Pricing under rough volatility},
  year      = {2016},
  number    = {6},
  pages     = {887-904},
  volume    = {16},
  publisher = {Routledge},
}

@Article{affine_fractional_stochastic_volatility_models,
  AUTHOR = {Comte, F. and Coutin, L. and Renault, E.},
     TITLE = {Affine fractional stochastic volatility models},
   JOURNAL = {Ann. Finance},
  FJOURNAL = {Annals of Finance},
    VOLUME = {8},
      YEAR = {2012},
    NUMBER = {2-3},
     PAGES = {337--378},
      ISSN = {1614-2446,1614-2454},
   MRCLASS = {60H30 (60G22 62M09 62M15 91B70 91G20)},
  MRNUMBER = {2922801},
MRREVIEWER = {Ludger\ Overbeck},
}

@article{short_time_fukasawa,
author = {Masaaki Fukasawa},
title = {Short-time at-the-money skew and rough fractional volatility},
journal = {Quantitative Finance},
volume = {17},
number = {2},
pages = {189-198},
year  = {2017},
publisher = {Routledge},
}

@book{ode_walter,
    author = {Wolfgang Walter},
    title = {Ordinary Differential Equations},
    publisher = {Springer New York, NY},
    year = {1998}
}

@Article{tudor_comparison,
  author    = {Constantin Tudor},
  journal   = {The Annals of Probability},
  title     = {A Comparison Theorem for Stochastic Equations with Volterra Drifts},
  year      = {1989},
  issn      = {00911798},
  number    = {4},
  pages     = {1541--1545},
  volume    = {17},
  publisher = {Institute of Mathematical Statistics},
}

@Article{sundar_comparison,
  author    = {Guillermo Ferreyra and Padamanbhan Sundar},
  journal   = {Bernoulli},
  title     = {Comparison of Stochastic Volterra Equations},
  year      = {2000},
  issn      = {13507265},
  number    = {6},
  pages     = {1001--1006},
  volume    = {6},
}

@book{mainardi_book,
    author = {Rudolf Gorenflo and  Anatoly A. Kilbas and Francesco Mainardi and Sergei Rogosin},
    title = {Mittag-Leffler Functions, Related Topics and Applications},
    publisher = {Springer Berlin, Heidelberg},
    year = {2020}
}

@article{proemel_scheffels_weak,
    author = {Pr\"omel, David J. and Scheffels, David},
    title = {On the existence of weak solutions to stochastic Volterra equations},
    journal = {Electron. Commun. Probab. 28, 1-12, (2023)},
    year = {2023}
}

@article{hamaguchi2023weak,
    author = {Yushi Hamaguchi},
    title = {{Weak well-posedness of stochastic Volterra equations with completely monotone kernels and nondegenerate noise}},
    volume = {35},
    journal = {The Annals of Applied Probability},
    number = {2},
    publisher = {Institute of Mathematical Statistics},
    pages = {1442 -- 1488},
    year = {2025},
}

@article{K21,
author = {Alexander Kalinin},
title = {{Support characterization for regular path-dependent stochastic Volterra integral equations}},
volume = {26},
journal = {Electronic Journal of Probability},
number = {none},
publisher = {Institute of Mathematical Statistics and Bernoulli Society},
pages = {1 -- 29},
year = {2021},
doi = {10.1214/20-EJP576},
URL = {https://doi.org/10.1214/20-EJP576}
}

\end{document}